\documentclass[11pt, a4paper]{amsart}
\pdfoutput=1

\usepackage[ansinew]{inputenc}
\usepackage[all]{xy}
\SelectTips{cm}{}
\usepackage[pagebackref,
  ,pdfborder={0 0 0}
  ,urlcolor=black,a4paper,hypertexnames=false]{hyperref}
\hypersetup{pdfauthor={Benjamin Br\"uck, Francesco Fournier-Facio, Clara L\"oh}
  ,pdftitle=Median quasimorphisms on CAT(0) cube complexes and their cup products}

\usepackage{amsfonts}
\usepackage{amssymb, amsmath,amsthm, mathtools}
\usepackage{tabularx, hyperref}
\usepackage{enumerate}

\usepackage{tikz-cd}

\usepackage{stmaryrd}

\usepackage{subcaption}
\usepackage{verbatim}
\usepackage[capitalize,nameinlink,noabbrev]{cleveref}

\makeatletter
\newtheorem*{rep@theorem}{\rep@title}
\newcommand{\newreptheorem}[2]{%
\newenvironment{rep#1}[1]{%
 \def\rep@title{#2 \ref{##1}}%
 \begin{rep@theorem}}%
 {\end{rep@theorem}}}
\makeatother

\newtheorem{lemma}{Lemma}[section]
\newtheorem{thm}[lemma]{Theorem}
\newreptheorem{thm}{Theorem}  
\newtheorem{prop}[lemma]{Proposition}
\newreptheorem{prop}{Proposition}  
\newtheorem{cor}[lemma]{Corollary}
\newreptheorem{cor}{Corollary}  

\theoremstyle{definition}
\newtheorem{defi}[lemma]{Definition}

\newtheorem{quest}[lemma]{Question}
\newreptheorem{quest}{Question}  

\newtheorem{example}[lemma]{Example}
\newtheorem{rem}[lemma]{Remark}

\newtheorem{setup}[lemma]{Setup}

\theoremstyle{definition}

\makeatletter
\newcommand\norm{\bBigg@{0.8}}

\makeatother
 
 \newcommand{\indnorm}[2][flex]{\csname #1l\endcsname\|#2%
                                 \csname #1r\endcsname\|\mathclose{}}
                                  \newcommand{\indnorml}[4][flex]{\csname #1l\endcsname\|#2%
                                 \csname #1r\endcsname\|_{#3}^{#4}\mathclose{}}

\newcommand{\genrel}[3][flex]{\csname #1l\endcsname\langle #2 \mathbin{\csname #1m\endcsname|} #3\csname #1r\endcsname\rangle}

\DeclareMathOperator{\Lex}{\sf{Lex}}

\DeclareMathOperator{\QM}{QM}
\DeclareMathOperator{\Map}{Map}

\DeclareMathOperator{\comp}{comp}

\DeclareMathOperator{\id}{id}

\DeclareMathOperator{\cat0}{CAT(0)}

\newcommand{\HH}{\operatorname{H}}

\newcommand{\CC}{\operatorname{C}}

\def\linf{\ell^{\infty}}

\newcommand{\es}{\varepsilon_s}
\newcommand{\er}{\varepsilon_r}

\newcommand{\h}{\mathcal{H}}
\newcommand{\g}{g}

\newcommand{\N}{\ensuremath {\mathbb{N}}}
\newcommand{\R} {\ensuremath {\mathbb{R}}}

\newcommand{\Z} {\ensuremath {\mathbb{Z}}}

\renewcommand{\rho}{\varrho}
\def\phi{\varphi}
\def\lone{\ell^1}

\def\args{\;\cdot\;}

\def\actson{\curvearrowright}

\newcommand{\ls}{\left\{}
\newcommand{\rs}{\right\}}
\newcommand{\median}{{m}}

\newcommand{\signedlabels}{\lambda^\pm}
\newcommand{\len}{\operatorname{len}}

\def\vertx#1{%
  \fill #1 circle (0.07);}
\def\hollowvertx#1{%
  \filldraw[fill=white] #1 circle (0.07);}

\def\hedgecolor{%
  purple}
\def\hedge#1#2{%
  \begin{scope}[\hedgecolor]
    \draw[line width={0.07cm}] #1 -- #2;
    \vertx{#1}
    \vertx{#2}
  \end{scope}}

\def\quartertree{%
  \begin{scope}[shift={(120:1)}]
    \vertx{(0,0)}
    \draw (0,0) -- +(75:0.5);
    \draw (0,0) -- +(165:0.5);
  \end{scope}
}
\def\halftree{%
  \begin{scope}[shift={(-0.5,0)}]
    \draw (0,0) -- +(0.5,0);
    \draw (0,0) -- +(120:1);
    \draw (0,0) -- +(240:1);
    \vertx{(0,0)}
    \quartertree
    \begin{scope}[rotate={120}]
      \quartertree
    \end{scope}
  \end{scope}
}
\def\tree{%
  \halftree
  \begin{scope}[rotate={180}]
    \halftree
  \end{scope}
}

\def\squaregrid{%
  \foreach \j in {-1.5,-0.5,0.5,1.5} {%
    \draw (\j,-1.5) -- (\j,1.5);
    \foreach \k in {-1,0,1} {%
      \vertx{(\j,\k)}
    }
  }
  \foreach \k in {-1,0,1} {%
    \draw (-2,\k) -- (2,\k);
  }
}

\def\halfspacecolor{black!30!blue}
\def\shalfspacecolor{yellow!20!red}
\def\halfspaceshade{%
  \begin{scope}[color=\halfspacecolor, opacity=0.1, line width=1mm]
    \clip (-2,0) rectangle +(4,1);
    \foreach \j in {-2.6,-2.35,...,2.6} {%
      \draw (\j,0.2) -- +(60:1);
    }
  \end{scope}}

\def\halfspaces#1#2#3{%
  \begin{scope}[\halfspacecolor]
  \begin{scope}[shift={#1}]
    \begin{scope}[rotate={#2}]
      \draw (-2,0.2) -- (2,0.2);
      \halfspaceshade
      \draw (-2,-0.2) -- (2,-0.2);
      \begin{scope}[shift={(0,-1.2)}]
        \halfspaceshade
      \end{scope}
      \draw (2.1,0.4) node {\smash{$#3$}};
      \draw (2.1,-0.4) node {\smash{$\overline{#3}$}};
    \end{scope}
  \end{scope}
  \end{scope}
}

\def\halfspace#1#2#3{%
  \begin{scope}[\halfspacecolor]
  \begin{scope}[shift={#1}]
    \begin{scope}[rotate={#2}]
      \draw (-2,0.2) -- (2,0.2);
      \halfspaceshade
      \draw (2.1,0.4) node {\smash{$#3$}};
    \end{scope}
  \end{scope}
  \end{scope}
}

\def\geodesiccolor{blue!50!green}

\usepackage{color}
\usepackage{pdfcolmk}

\long\def\forget#1{}

\def\longrightarrow{\rightarrow}
\def\longmapsto{\mapsto}
\def\widetilde{\tilde}

\makeatletter
\@namedef{subjclassname@2020}{%
  \textup{2020} Mathematics Subject Classification}
\makeatother

\begin{document}

\title[Median quasimorphisms and their cup products]%
      {\makebox[0pt]{Median quasimorphisms on $\boldsymbol{\cat0}$ cube complexes}\\
        and their cup products}

\author[]{Benjamin Br\"uck}
\address{Department for Mathematical Logic and Foundational Research, University of M\"unster, Germany}
\email{benjamin.brueck@uni-muenster.de}

\author[]{Francesco Fournier-Facio}
\address{Department of Pure Mathematics and Mathematical Statistics, University of Cambridge, UK}
\email{ff373@cam.ac.uk}

\author[]{Clara L\"oh}
\address{Fakult\"{a}t f\"{u}r Mathematik, Universit\"{a}t Regensburg, Regensburg, Germany}
\email{clara.loeh@ur.de}

\thanks{}


\date{\today.\ 
  Benjamin Br\"uck was supported by the Deutsche Forschungsgemeinschaft (DFG, German Research Foundation) under Germany's Excellence Strategy EXC 2044 - 390685587, Mathematics M\"unster: Dynamics-Geometry-Structure.  
  Francesco Fournier-Facio has been supported by the Herchel Smith Postdoctoral Fellowship Fund.
  Clara L\"oh has been supported by the CRC~1085 \emph{Higher Invariants}
  (Universit\"at Regensburg, funded by the~DFG)}

\phantom{.}
\vspace{-\baselineskip}

\maketitle

\vspace{-1.2\baselineskip}

\begin{abstract}
  Cup products provide a natural approach to access higher bounded
  cohomology groups.  We extend vanishing results on cup products of
  Brooks quasimorphisms of free groups to cup products of median
  quasimorphisms, i.e., Brooks-type quasimorphisms of group actions on
  $\cat0$ cube complexes. In particular, we obtain such vanishing
  results for groups acting on trees and for right-angled Artin
  groups.  Moreover, we outline potential applications of vanishing
  results for cup products in bounded cohomology.
\end{abstract}

\section{Introduction}

The aim of this article is to extend vanishing results for cup
products in the bounded cohomology of free groups to group actions on
$\cat0$ cube complexes. We first explain the context of these results.

Bounded cohomology of groups is the cohomology of the complex of
bounded equivariant functions on the simplicial resolution
(Section~\ref{s:bc}). The boundedness condition enables exciting
applications to the geometry of manifolds~\cite{vbc}, stable
commutator length~\cite{Calegari}, group actions on the
circle~\cite{Ghys}, and rigidity theory~\cite{rigidity}. However,
boundedness causes a lack of excision; therefore, despite a
functional-analytic version of the resolution calculus~\cite{Ivanov},
bounded cohomology remains difficult to compute in general.

A key example is given by non-abelian free groups~$F$. The bounded
cohomology of (the one-dimensional groups!)~$F$ is known to be
non-trivial in degrees~$2$ (via quasimorphisms, Section~\ref{ss:qm})
and~$3$ (via hyperbolic geometry~\cite{Soma}). However, it is unknown
whether the higher-degree bounded cohomology~$\HH^n_b(F;\R)$
for~$n\geq 4$ is non-trivial~\cite[Question~16.3]{BIMW}. A related open problem is to decide
whether epimorphisms of groups induce injective maps on the level of
bounded cohomology in all degrees.  An affirmative answer to the
latter question would lead to non-triviality of the bounded cohomology
of non-abelian free groups in degrees~$\geq 4$, which in turn would
have consequences for large classes of geometrically defined 
groups~\cite{Hullosin,FPS}.

The cup product on bounded cohomology sheds some light on both
problems by producing classes in high degrees from lower degrees (for
the second problem, this is explained in Appendix~\ref{appx:lex}). A
natural starting point are non-trivial classes in degree~$2$: Bounded
cohomology in degree~$2$ is closely related to quasimorphisms, which
are relevant in geometric group theory~\cite{Fujiwara:mcg}, knot
theory~\cite{Malyutin}, and symplectic geometry~\cite{symplectic}.  More
precisely, if $f \colon \Gamma \to \R$ is a quasimorphism of a
group~$\Gamma$, then the simplicial cochain~$\widehat f$ associated
with~$f$ defines a class~$[\delta^1 \widehat f]$
in~$\HH^2_b(\Gamma;\R)$ (Section~\ref{ss:qm}).

\subsection{Brooks quasimorphisms}

For free groups, all bounded cohomology classes in degree~$2$ come
from quasimorphisms, and a key example of quasimorphisms on free groups
are Brooks quasimorphisms~\cite{Brooks} (Example~\ref{exa:brooks}).
The cup product on the bounded cohomology of free groups was studied
independently by Bucher--Monod~\cite{BM:cup} and
Heuer~\cite{Heuer:cup}, who proved that many cup products of
quasimorphism classes are trivial. Recently, Amontova--Bucher combined
these techniques:

\begin{thm}[Amontova--Bucher~\cite{AB:cup}]
  Let $F$ be a free group, let $w \in F$ and let $H_w$ be the
  corresponding (big) Brooks quasimorphism. Then, for every~$n \geq 1$ and every~$\zeta \in
  \HH^n_b(F; \R)$, the cup
  product~$[\delta^1 \widehat{H_w}] \cup \zeta$ is trivial
  in~$\HH_b^{2+n}(F;\R)$.
\end{thm}

These results partially generalise to other types of quasimorphisms on
free groups~\cite{Heuer:cup,Fournier:cup,AB:cup}, but it is not clear
whether they hold for \emph{all} quasimorphisms on free groups.  In
the present paper, we will focus on Brooks quasimorphisms and their
generalisations.  Brooks quasimorphisms and their bounded cohomology
classes have been generalised in many
directions~\cite{Fujiwara:hyperbolic, Fujiwara:hyperbolic2,
  Fujiwara:amalgamated, Grigorchuk:amalgamated, Fujiwara:mcg, FFT}.
Thus, the following (vague) question arises:

\begin{quest}\label{q}
  Which Brooks-type quasimorphisms produce trivial cup products in
  bounded cohomology?
\end{quest}

For example, this was recently answered for de~Rahm
quasimorphisms on surface groups, a continuous analogue of
Brooks quasimorphisms~\cite{surface:cup}.

In the present paper, we answer Question~\ref{q} for \emph{median
  quasimorphisms}, which are Brooks-type quasimorphisms, defined for
groups acting on $\cat0$ cube complexes.  We use the framework of
equivariant bounded cohomology~\cite{Loehsauer, Li} to state our
results in a more general, vastly applicable way.

\begin{rem}
  Before moving forward with the statements of results, we point
  out that cup products in bounded cohomology can very well be
  non-trivial.  E.g., this happens for cup powers of the
  Euler class of Thompson's group~$T$ \cite{rigidity, cohoT,binate,
    Monodnariman} and for certain direct
  products~\cite{clara,bcfp}. In particular, this can occur for
  groups acting on $\cat0$ cube complexes (Example~\ref{ex:FxF}).
\end{rem}

\subsection{Groups acting on trees}

We start by stating our results for groups acting on trees, where the
statements are stronger and parallel the case of the free group more
closely.  Given an action of a group~$\Gamma$ on a simplicial
tree~$T$, an oriented geodesic segment~$s$ in~$T$, and a base
vertex~$x$, one associates a counting quasimorphism, the \emph{median
  quasimorphism}~$f_{s, x} \colon \Gamma \to \R$ (see
Section~\ref{s:trees} for the relevant definitions). These
quasimorphisms were defined by Monod and Shalom~\cite{Monodshalom} and
are a dynamical analogue of Brooks quasimorphisms: If $\Gamma$ is a
free group, $T$ is a Cayley tree of~$\Gamma$, and $s$ is the geodesic
segment~$[e, w]$ in~$T$ for~$w \in \Gamma$, then the median
quasimorphism~$f_{s,e}$ coincides with the Brooks quasimorphism~$H_w$.
These quasimorphisms have proven very useful in the study of rigidity
properties of groups acting on trees~\cite{Monodshalom, Monodshalom2}
and they admit an easy-to-check dynamical criterion for
triviality~\cite{IPS}.  Since the equivariant point of view is more
natural for our purposes, we consider \emph{median
  quasimorphisms}~$f_s$ of the action~$\Gamma \actson T$, which
dispense of the choice of base vertex
(Definition~\ref{defi:qmaction}).

Our first main result is the answer to Question~\ref{q} for median
quasimorphisms on groups acting on trees:

\begin{thm}[Theorem~\ref{thm:maintree}]\label{intro:thm:maintree}
  Let $\Gamma \actson T$ be a group action on a tree~$T$.  Let $s$ be
  an oriented geodesic segment in~$T$, and let $f_s$ be the
  corresponding median quasimorphism of~$\Gamma \actson T$.  Then, for
  every~$n \geq 1$ and every~$\zeta \in \HH^n_{\Gamma, b}(T; \R)$, the
  cup product~$[\delta^1 f_s] \cup \zeta \in
  \HH_{\Gamma,b}^{2+n}(T;\R)$ is trivial.
\end{thm}

Here $\HH^n_{\Gamma, b}(T; \R)$ denotes the $\Gamma$-equivariant
bounded cohomology of~$T$, which we tacitly identify with its vertex
set.  Upon choosing a base vertex $x \in T$ we obtain the median
quasimorphism $f_{s, x}$, and then Theorem~\ref{intro:thm:maintree}
states that the corresponding class~$[\delta^1 \widehat{f_{s, x}}] \in
\HH^2_b(\Gamma; \R)$ has trivial cup product with every class that can
also be described via the action of~$\Gamma$ on~$T$.  In the case of
amenable stabilisers, all classes in~$\HH^*_b(\Gamma;\R)$ are of this
form:

\begin{cor}[Corollary~\ref{cor:maintreegroup}]
\label{intro:cor:maintree}
  With the notation of Theorem~\ref{intro:thm:maintree}, suppose that
  vertex stabilizers are amenable and let $x$ be a base vertex.  Then,
  for every~$n \geq 1$ and every~$\zeta \in \HH^n_b(\Gamma; \R)$, the
  cup product~$[\delta^1 \widehat{f_{s, x}}] \cup \zeta \in
  \HH_b^{2+n}(\Gamma;\R)$ is trivial.
\end{cor}

Via Bass--Serre Theory, the class of groups to which
Corollary~\ref{intro:cor:maintree} applies coincides with fundamental
groups of graphs of groups with amenable vertex groups (see Remark
\ref{rem:gog}).  This includes the following examples, which
show that Theorem~\ref{intro:thm:maintree} applies much further than
to free groups:
  virtually free groups,
  Baumslag--Solitar groups,
  torus knot groups,
  generalized Baumslag--Solitar groups, 
  amalgamated products and HNN-extensions of amenable groups.

While our proof is based on the same blueprint as the proofs by
Bucher, Monod, Heuer, and Amontova, it has the advantage of avoiding
heavyweight combinatorics and the complex of aligned chains. In
particular, it admits a straightforward generalisation to actions on
$\cat0$ cube complexes (Section~\ref{ss:intro:ccc}). The aligned
cochain complex was introduced by Bucher and Monod~\cite{aligned}, and
it significantly simplifies computations in bounded cohomology of
groups acting on trees~\cite{BM:cup,AB:cup}. However, it is currently
unknown whether this complex can be generalised suitably to actions on
other objects~\cite{aligned}. Heuer's approach~\cite{Heuer:cup} is
based on involved combinatorics and might be cumbersome to adapt to
actions on higher-dimensional objects.

\subsection{Groups acting on $\boldsymbol{\cat0}$ cube complexes}\label{ss:intro:ccc}

Now we let $\Gamma$ act on a finite-dimensional $\cat0$ cube
complex~$X$.  One can generalize oriented geodesic segments in trees
to \emph{$\h$-segments}~\cite{CFI, FFT}, i.e., to sequences of tightly
nested halfspaces in~$X$.  We use such sequences to define
\emph{median quasimorphisms} on~$\Gamma$ analogously to the case of
trees (see Section \ref{s:ccc} for the relevant definitions).  These
are the overlapping version of the effective quasimorphisms on RAAGs
defined by Fern{\'o}s--Forrester--Tao~\cite{FFT}.  The generalization
is more subtle than it may look at first sight: Unlike the
non-overlapping case~\cite{FFT}, the natural definition does not lead
to a quasimorphism in general (Examples~\ref{exa:dunbounded_dim} and~\ref{exa:dunbounded_staircase}).
We show that we do obtain quasimorphisms under our running
assumption that the $\cat0$ cube complex~$X$ is finite-dimensional
and has finite staircase length (Proposition~\ref{prop:fs:qm}).

The situation for cup products is more delicate than in the case of
trees: We exhibit median quasimorphisms whose cup product is
\emph{not} trivial (Example~\ref{ex:FxF}). However, we can 
give sufficient conditions that ensure that the cup product vanishes.
Once again, we state the result for the corresponding median classes
in equivariant bounded cohomology; the definitions of the
quasimorphisms~$f_s$ and of non-transversality are given in
Section~\ref{ss:cupccc}.

\begin{thm}[Theorems~\ref{thm:mainccc} and~\ref{thm:cuptwomedian}]\label{intro:thm:mainccc}
  Let $\Gamma$ be a group acting on a finite-dimensional $\cat0$ cube
  complex~$X$ with finite staircase length.
  Let $s$ be an $\h$-segment in~$X$, and let $f_s$ be
  the corresponding median quasimorphism of~$\Gamma \actson X$.
  Then, for every class $\zeta \in
  \HH^n_{\Gamma, b}(X; \R)$ that is non-transverse to the orbit~$\Gamma s$, 
  the cup product~$[\delta^1 f_s] \cup \zeta \in \HH_{\Gamma,b}^{2+n}(X;\R)$
  is trivial.

  In particular, this holds if $\zeta = [\delta^1 f_r]$ for another
  $\h$-segment~$r$ and the first and last halfspaces of $s$ and~$r$
  have non-transverse orbits.
\end{thm}

As previously mentioned, we use the hypotheses that~$X$ is
finite-di\-men\-sion\-al and has finite staircase length to ensure
that~$f_s$ be a quasimorphism. We use them also crucially in the proof,
in order to guarantee that the explicit primitive we construct is
bounded.

As a concrete example, we examine the case of right-angled Artin groups:

\begin{cor}[Corollary~\ref{cor:RAAG:twomedian}]\label{intro:cor:RAAGs}
 Let $\Gamma = A(G)$ be a RAAG acting on the universal covering~$X$ of
 its Salvetti complex with base vertex~$x$.  Let $s$,~$r$ be
 $\h$-segments in~$X$ such that the labels of the first and last
 halfspaces in~$s$
 and the labels of the first and last halfspaces in~$r$
 are not connected by an edge in~$G$. Then
 $[\delta^1 \widehat{f_{s, x}}] \cup [\delta^1 \widehat{f_{r, x}}]
 = 0 \in \HH^4_b(\Gamma; \R)$.
\end{cor}

There is a wealth of such genuinely median classes: We show that many
of these classes~$[\delta^1 \widehat{f_{s,x}}]$ are non-trivial in
bounded cohomology (Section~\ref{ss:nontrivialmedian}) and that the
median quasimorphisms~$f_{s,x}$ are not at bounded distance to
canonical pullbacks of Brooks quasimorphisms
(Section~\ref{ss:pullbackbrooks}).  We hope that these classes will
shed new light on the bounded cohomology of RAAGs and their subgroups,
potentially leading to an understanding analogous to that of
quasimorphisms of free groups via Brooks
quasimorphisms~\cite{Grigorchuk, counting, Fournier:cup}.

Recently, Marasco proved that for many quasimorphism classes of free
groups also higher order vanishing occurs: The corresponding Massey
triple products in bounded cohomology are zero~\cite{marasco_massey}.
In the meantime, the techniques from the present article have been
used to show triviality of Massey triple products in the context
of median quasimorphisms~\cite{fhofmann_msc}.

Finally, in addition to the application of cup product vanishing in
bounded cohomology outlined in Appendix~\ref{appx:lex}, let us mention
the following open problem: It is unknown whether -- in analogy with
the cup length bound for the Lusternik--Schnirelmann category -- the
cup length for the bounded cohomology is a lower bound for the
amenable category~\cite[Remark~3.17]{CLM}. Again, an interesting test
case is given by non-abelian free groups, relating this question to
the above problems on the bounded cohomology of free groups.

\subsection*{Organization of the paper}

Basics on bounded cohomology, equivariant bounded cohomology, and
quasimorphisms are collected in Section~\ref{s:bc}. Median
quasimorphisms for actions on $\cat0$ cube complexes and their cup
products in bounded cohomology are studied in Section~\ref{s:ccc}.  We
specialise to the tree case in Section~\ref{s:trees} and to the
right-angled Artin case in Section~\ref{s:raag}.
Appendix~\ref{appx:lex} outlines the connection to the $\Lex$-problem.

\subsection*{Acknowledgements}

We wish to thank Matthew Cordes, Talia Fern{\'o}s, Elia Fioravanti and Nicolas Mo\-nod for helpful
discussions, as well as Sofia Amontova,
Michelle Bucher, and Marco Moraschini for discussions on
cup products and~$\Lex$.
Moreover, we would like to thank Franziska Hofmann for spotting a gap
in an earlier version of this paper, and the anonymous referee for
several helpful comments, including a fix of the gap.

\section{Bounded cohomology}\label{s:bc}

We recall basic notions concerning bounded cohomology and
quasimorphisms and refer the reader to the literature for more
details and proofs~\cite{vbc,Ivanov,Monod,Frigerio}.
We will only be interested in the case of discrete groups
and trivial real coefficients.

\subsection{Bounded cohomology of groups}

Let $\Gamma$ be a discrete group.  We consider the
complex~$\linf(\Gamma^{*+1},\R)$ of normed $\R[\Gamma]$-modules,
equipped with the simplicial coboundary operators.  The subcomplex of
$\Gamma$-invariants is denoted by~$\CC^*_b(\Gamma; \R) \coloneqq
\linf(\Gamma^{*+1},\R)^\Gamma$.  The \emph{bounded cohomology}
of~$\Gamma$ (with trivial real coefficients) is defined as
\[ \HH_b^*(\Gamma;\R) \coloneqq \HH^* \bigl( \CC^*_b(\Gamma;\R)\bigr).  
\]

The canonical inclusion~$\ell^\infty(\Gamma^{*+1};\R) \hookrightarrow \Map(\Gamma^{*+1},\R)$
from bounded to ordinary cochains
induces a natural transformation between bounded cohomology and
ordinary cohomology, the \emph{comparison map}
\[ \comp_\Gamma^* \colon \HH_b^*(\Gamma;\R) \to \HH^*(\Gamma;\R).
\]
The kernel of~$\comp_\Gamma^2$ is related to quasimorphisms (Section~\ref{ss:qm}).

The usual formula gives a cup product on bounded cohomology:
\begin{align*}
  \cup \colon \CC_b^p (\Gamma;\R) \otimes_\R \CC_b^q(\Gamma;\R) 
  & \to \CC_b^{p+q} (\Gamma;\R)
  \\
  f \otimes g
  & \mapsto
  \bigl( (\gamma_0,\dots, \gamma_{p+q}) \mapsto
  f(\gamma_0,\dots,\gamma_p) \cdot g(\gamma_p,\dots,\gamma_{p+q})\bigr)
  \\
  \cup \colon \HH_b^p (\Gamma;\R) \otimes_\R \HH_b^q(\Gamma;\R)
  & \to \HH_b^{p+q}(\Gamma;\R)
  \\
  [f] \otimes [g]
  & \mapsto [f \cup g]  
\end{align*}
These operations are well-defined and satisfy the usual equations with respect
to the simplicial coboundary operator.
In particular, we have 
\[\delta^{p+q}(f \cup g)
= \delta^p f \cup g + (-1)^p \cdot f \cup \delta^q g.
\]

\subsection{Equivariant bounded cohomology}
\label{ss:equivariant}

Equivariant bounded cohomology is a bounded version
of equivariant cohomology~\cite{Li,Loehsauer}.

\begin{defi}
  Let $\Gamma \actson S$ be an action of a group~$\Gamma$ on a set~$S$.
  We write
  \[ \CC_{\Gamma,b}^*(S;\R) \coloneqq \ell^\infty(S^{*+1},\R)^\Gamma,
  \]
  where $\ell^\infty(S^{*+1},\R)$ carries the simplicial coboundary
  operator and the $\Gamma$-action given by
  \[ (s_0,\dots,s_n) \mapsto f(\gamma^{-1}\cdot s_0, \dots, \gamma^{-1}\cdot s_n)
  \]
  for all~$\gamma \in \Gamma$, $f \in \ell^\infty(S^{n+1},\R)$.
  The \emph{$\Gamma$-equivariant bounded cohomology} of~$S$
  (with trivial real coefficients) is defined as
  \[ \HH_{\Gamma,b}^* (S;\R)
     \coloneqq \HH^*\bigl( \CC_{\Gamma,b}^*(S;\R)\bigr).
  \]
  For an action~$\Gamma \actson X$ on a $\cat0$ cube complex with vertex set~$V$
  (Section~\ref{ss:cccprelim}),
  we will also write
  \[ \HH_{\Gamma,b}^*(X;\R) \coloneqq \HH_{\Gamma,b}^*(V;\R).
  \]
\end{defi}

The usual formula gives a cup product on equivariant bounded cohomology:
Let $\Gamma \actson S$ be a group action and let $p,q\in \N$. Then
\begin{align*}
  \cup \colon \CC_{\Gamma,b}^p (S;\R) \otimes_\R \CC_{\Gamma,b}^q(S;\R) 
  & \to \CC_{\Gamma,b}^{p+q} (S;\R)
  \\
  f \otimes g
  & \mapsto
  \bigl( (s_0,\dots, s_{p+q}) \mapsto
  f(s_0,\dots,s_p) \cdot g(s_p,\dots,s_{p+q})\bigr)
  \\
  \cup \colon \HH_{\Gamma,b}^p (S;\R) \otimes_\R \HH_{\Gamma,b}^q(S;\R)
  & \to \HH_{\Gamma,b}^{p+q}(S;\R)
  \\
  [f] \otimes [g]
  & \mapsto [f \cup g]  
\end{align*}
are well-defined and satisfy the usual equations with respect
to the simplicial coboundary operator.

Equivariant and ordinary bounded cohomology are related via pullbacks
of orbit maps.  Namely for every~$x \in S$, the \emph{orbit map}~$o_x
\colon \Gamma \to S, \gamma \mapsto \gamma \cdot x$ induces a 
cochain map~$\CC^*_{\Gamma, b}(S; \R) \to \CC_b^*(\Gamma; \R)$.
We denote the induced map in bounded cohomology by
\[o_x^* \coloneqq \HH^*_{\Gamma,b}(o_x;\R) \colon
\HH^*_{\Gamma, b}(S; \R) \to \HH^*_b(\Gamma; \R). \]
The map~$o_x^*$ is compatible with the cup products.

In general, this map~$o_x^*$ is neither injective nor surjective. It is
an isomorphism under some additional hypotheses on the stabilizers:

\begin{thm}[{\cite[Theorem 4.23]{Frigerio}}]
  \label{thm:amenable:stab}
  Let $\Gamma \actson S$ be a group action with amenable stabilizers.
  Then for every~$x \in S$, the orbit map
  \[ o_x^* \colon \HH^*_{\Gamma, b}(S; \R) \to \HH^*_b(\Gamma; \R)\]
  is an isomorphism.
\end{thm}

Similarly, one can also consider the case of uniformly boundedly acyclic
actions~\cite[Section~5.5]{Liloehmoraschini}.

\subsection{Quasimorphisms}
\label{ss:qm}

The exact part of bounded cohomology in degree~$2$ can be described
in terms of quasimorphisms.

\begin{defi}[quasimorphism]\label{def:qmgroup}
  Let $\Gamma$ be a group. A \emph{quasimorphism of~$\Gamma$} is
  a function~$f \colon \Gamma \to \R$ whose \emph{defect}
  \[ D(f)
  \coloneqq \sup\limits_{\gamma,\lambda \in \Gamma}
  \bigl|f(\gamma) + f(\lambda) - f(\gamma\cdot \lambda)\bigr|
  \]
  is finite.  A quasimorphism is called \emph{trivial} if it is a sum
  of a homomorphism and a bounded function.
\end{defi}

We recall an important example of quasimorphisms of the free group:
the Brooks quasimorphisms~\cite[2.3.2]{Calegari}. These
counting quasimorphisms are at the basis of the generalisations
explored in this paper.

\begin{example}[Brooks quasimorphisms]\label{exa:brooks}
  Let $F$ be a free group with a free generating set~$S$ 
  and let $w \in F$ be a reduced word over~$S \cup S^{-1}$.
  For an element~$\gamma \in F$, we denote by~$C_w(\gamma)$ the number
  of copies of~$w$ that appear in a reduced expression for~$\gamma$.
  The \emph{(big) Brooks quasimorphism for~$w$} is
  \begin{align*}
    H_w \colon F & \to \R\\
    \gamma & \mapsto C_w(\gamma) - C_{w^{-1}}(\gamma).
  \end{align*}
  Then $D(H_w) \leq 3 \cdot (|w| - 1)$, where $|w|$ denotes the length
  of~$w$.

  Another option is to define~$c_w(\gamma)$ as the maximal number of
  \emph{non-over\-lap\-ping} copies of~$\gamma$ that appear in a reduced
  expression for~$\gamma$.  This leads to the \emph{(small) Brooks
    quasimorphism for~$w$}, denoted by~$h_w$.  The advantage of the
  small Brooks quasimorphism is that $D(h_w) \leq 2$, which gives a
  bound independent of~$w$. This makes~$h_w$ more suitable for
  quantitative applications such as computations of stable commutator
  length.  If $w$ is \emph{non-self-overlapping} (i.e., there is no reduced
  expression~$w = uvu$ with a non-empty word~$u$), then $H_w = h_w$.
\end{example}

For later reference, we note the following property of Brooks quasimorphisms:

\begin{lemma}
\label{lem:bounded_distance_brooks}
  Let~$F$ be a non-abelian free group and~$w, w'\in F$. Then~$H_w$ is
  at bounded distance from~$H_{w'}$ if and only if~$w = w'$.
\end{lemma}

Note that~$F$ must be non-abelian for this to hold. If~$F
= \langle a \rangle \cong \mathbb{Z}$, then~$H_{a^n}(a^k) =
\mathrm{sign}(k n) \cdot (|k| - |n| + 1)$ whenever $|k| \geq |n|$, so
every Brooks quasimorphism is close to the identity or to minus the
identity.

\begin{proof}
	Assume that $H_w$ is  at bounded distance from~$H_{w'}$. We need to show that this implies $w = w'$.

  Let~$S$ be a symmetrized basis of~$F$, so by assumption~$|S|
  \geq 4$.  Let~$a \in S$ be such that~$w$ does not start
  with~$a^{-1}$, and~$w'$ does neither start nor end with~$a$.
  Let~$a^{-1} \neq b \in S$ be such that~$w$ does not end
  with~$b^{-1}$, and~$w'$ does not end with~$b$.  Let~$k \geq 1$ be
  larger than the length of both~$w$ and~$w'$.  Then~$a^k w b^k$ is a
  cyclically reduced word, and we claim that~$\overline{H_w}(a^k w
  b^k) > 0$.  For this, it suffices to notice that~$w^{-1}$ cannot be
  a subword of~$(a^k w b^k)^n$ for any~$n \geq 1$.  Indeed, no
  occurrence of~$w^{-1}$ can overlap with~$w$~\cite[Lemma
    3.14]{Fournier:cup}, and also by assumption~$w^{-1}$ cannot end
  with~$a$ (because~$w$ does not start with~$a^{-1}$) and~$w^{-1}$
  cannot start with~$b$ (because~$w$ does not end with~$b^{-1}$).

  It follows that~$\overline{H_{w'}}(a^k w b^k) > 0$ as well, in
  particular~$w'$ must occur as a subword of~$(a^k w b^k)^n$ for
  some~$n>0$.  However, $w'$ does neither start nor end with~$a$, and
  it does not end with~$b$.  Since moreover~$k$ is larger than the
  length of~$w'$, we conclude that~$w'$ must occur as a subword
  of~$w$.

  Swapping the roles of~$w$ and~$w'$, we see that $w$ must occur as a
  subword of~$w'$, which completes the proof.
\end{proof}

Every quasimorphism~$f \colon \Gamma \to \R$ defines a
$\Gamma$-invariant $1$-cochain via 
\[ \widehat f \colon (\gamma_0,\gamma_1)
\mapsto f(\gamma_0^{-1}\cdot \gamma_1).
\]
By construction, $\delta^1 \widehat f$ is bounded -- that is, $\widehat f$
is a \emph{quasicocycle} -- and thus $\delta^1 \widehat f$ defines a
class in~$\HH_b^2(\Gamma;\R)$. In fact, such classes form the
kernel of the comparison map in degree~$2$~\cite[Proposition 2.8]{Frigerio}:

\begin{prop}\label{prop:kercomp2}
  Let $\Gamma$ be a group. Then the sequence
  \begin{align*}
    0 \longrightarrow
    \QM(\Gamma) / \QM_0(\Gamma)
    & \longrightarrow
    \HH_b^2(\Gamma;\R)
    \overset{\comp_\Gamma^2}{\longrightarrow}
    \HH^2(\Gamma;\R)
    \\
    [f]
    & \longmapsto
    [\delta^1 \widehat f] 
  \end{align*}
  is exact, where $\QM(\Gamma)$ is the space of quasimorphisms on~$\Gamma$
  and $\QM_0(\Gamma)$ the subspace of trivial quasimorphisms.

  Moreover, every element in~$\QM(\Gamma)$ admits up to bounded functions
  a unique homogeneous representative (i.e., a quasimorphism that is a
  homomorphism on all cyclic subgroups).
\end{prop}

In particular, if the comparison map in degree~$2$ is trivial, then
every bounded cohomology class is represented by a quasimorphism.
This is of course the case for groups~$\Gamma$ with 
$\HH^2(\Gamma; \R) \cong 0$, e.g., for free groups. Therefore, the
study of quasimorphisms is central to the study of the bounded
cohomology of free groups. More generally, we have:

\begin{prop}[{\cite[Corollary 5.4]{CLM}, \cite[Example 4.7]{Li}}]
  Let $\Gamma$ be a group acting on a tree with amenable vertex stabilizers.
  Then the comparison map~$\comp^n_\Gamma$ is trivial for all~$n \geq 1$.
\end{prop}

\subsection{Quasimorphisms of group actions}

In view of the cocycle description of quasimorphisms,
we can speak of quasimorphisms of group actions: 

\begin{defi}[quasimorphism of a group action]\label{defi:qmaction}
  Let $\Gamma \actson S$ be a group action on a set~$S$. A
  \emph{quasimorphism of~$\Gamma\actson S$} is a function~$f \colon S \times S \to \R$
  that is $\Gamma$-invariant (with respect to the diagonal action on~$S \times S$)
  and has finite \emph{defect}
  \[ D(f) \coloneqq \| \delta^1 f\|_\infty
  = \sup_{x,y,z \in S} \bigl| f(y,z) - f(x,z) + f(x,y) \bigr|.
  \]
\end{defi}

\begin{rem}\label{rem:qmbc}
  If $f$ is a quasimorphism of a group action~$\Gamma \actson S$, then
  $\delta^1 f$ is a bounded cocycle in~$\CC^2_{\Gamma,b}(S;\R)$ and
  thus defines a class~$[\delta^1 f] \in \HH_{\Gamma,b}^2(S;\R)$.
  By construction, the set of all such classes coincides with the
  kernel of the comparison map~$\HH_{\Gamma,b}^2(S;\R) \to \HH_\Gamma^2(S;\R)$. 
\end{rem}

\begin{rem}[from group actions to quasimorphisms on groups]
\label{rem:qmtogroups}
  Let $f \colon S \times S \to \R$ be a quasimorphism of a group
  action~$\Gamma \actson S$. If $x \in S$, then
  \begin{align*}
    f_x \colon \Gamma & \to \R
    \\
    \gamma & \mapsto f(x,\gamma \cdot x)
  \end{align*}
  is a quasimorphism of~$\Gamma$ in the sense of
  Definition~\ref{def:qmgroup}.
  By construction,
  \[ [\delta^1 \widehat{f_x}]
  = o_x^*\bigl([\delta^1 f]\bigr)
  \in \HH_b^2(\Gamma;\R),
  \]
  where $o_x\colon \Gamma \to S$ is the orbit map of~$x$ and~$\widehat{f_x}$
  is the $\Gamma$-invariant $1$-cochain defined by~$f_x$.

  If $x$ and $x'$ lie in the same orbit, then the classes
  in~$\HH_b^2(\Gamma;\R)$ differ only by the conjugation action.
  The conjugation action is trivial on~$\HH_b^2(\Gamma;\R)$:
  in degree~$2$ this can be dervied from the exact sequence
  in Proposition~\ref{prop:kercomp2}, the triviality of
  the conjugation action on~$\HH^2(\Gamma;\R)$, and the
  fact that homogeneous quasimorphisms are conjugation
  invariant; more generally, one can prove the triviality
  of the conjugation action on bounded cohomology by
  a quantitative version of the proof for the classical
  case~\cite[Lemma~A.2]{clara}. Therefore, we
  obtain~$[\delta^1f_x] = [\delta^1 f_{x'}]$ in~$\HH_b^2(\Gamma;\R)$.
\end{rem}

Examples of quasimorphisms of group actions will be given in
Section~\ref{s:ccc} and Section~\ref{s:raag}.

\section{Median quasimorphisms}
\label{s:ccc}

We define median quasimorphisms on groups
acting on $\cat0$ cube complexes and their associated 
equivariant bounded cohomology classes. These classes do not have
vanishing cup products in general (Example~\ref{ex:FxF}), but we show
that they do under additional compatibility hypotheses.

\subsection{Preliminaries on $\boldsymbol{\cat0}$ cube complexes}\label{ss:cccprelim}

We recall the necessary basics on $\cat0$ cube complexes
and median graphs~\cite{Sageev2014, Roller, BH:Metricspacesnon}:

A cube complex~$X$ is \emph{non-positively curved} if the links of its
vertices are flag complexes. Informally, this means that as soon as
there is a corner of a possible cube of dimension at least~$3$ in~$X$,
there is an actual cube with this corner. We refer to this as the
\emph{link condition}. A \emph{$\cat0$ cube complex} is a simply
connected non-positively curved cube complex.

Throughout this section, we will consider the following situation:

\begin{setup}\label{setup:cat0}
Fix a $\cat0$ cube complex~$X$ with vertex set~$V$
and an action of a group~$\Gamma$ on~$X$ by combinatorial automorphisms.
We denote the $k$-skeleton of~$X$ by~$X(k)$; in particular, $X(0) = V$.
The $1$-skeleton~$X(1)$ is an undirected graph equipped with
the $\lone$-metric~$D$. Given~$x, y \in V$, we denote by~$[x, y]$ the
union of all combinatorial geodesics (represented by sequences of
vertices) from~$x$ to~$y$; such sets are called \emph{intervals}.
\end{setup}

The most important property of $X(1)$ is that it is a \emph{median
  graph}: for every triple of points~$x, y, z \in V$, the
intersection~$[x, y] \cap [y, z] \cap [x, z]$ consists of a unique
vertex, which we denote by~$m(x, y, z)$.

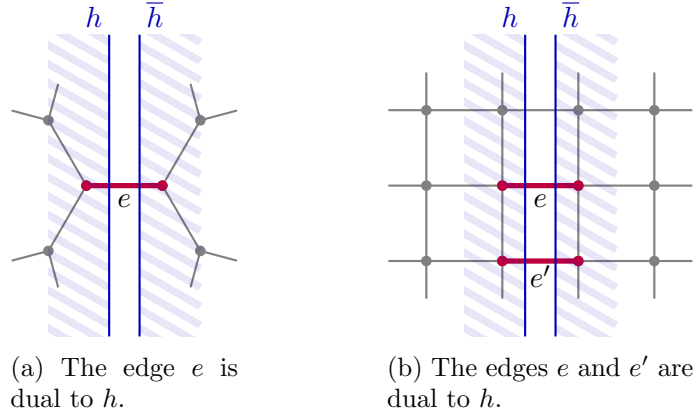
\begin{figure}
  \begin{center}
  \subcaptionbox{The edge~$e$ is dual to~$h$.}{%
    \begin{tikzpicture}[x=1cm,y=1cm,thick]
      \begin{scope}[black!50]
      \tree
      \end{scope}
      \hedge{(-0.5,0)}{(0.5,0)}
      \draw (0,-0.2) node{$e$};
      \halfspaces{(0,0)}{90}{h}
    \end{tikzpicture}}
  \hfil
  \subcaptionbox{The edges $e$ and~$e'$ are dual to~$h$.%
      \label{subfig:dualedges}}{%
      \begin{tikzpicture}[x=1cm,y=1cm,thick]
        \begin{scope}[black!50]
          \squaregrid
        \end{scope}
        \hedge{(-0.5,0)}{(0.5,0)}
        \draw (0,-0.2) node{$e$};
        \hedge{(-0.5,-1)}{(0.5,-1)}
        \draw (0,-1.2) node{$e'$};
        \halfspaces{(0,0)}{90}{h}
  \end{tikzpicture}}
  \end{center}
  
  \caption{Halfspaces in a tree and a square grid.}
  \label{fig:halfspace}
\end{figure}

For every edge~$e$ in~$X$, there exists a closest-point projection,
called the \emph{gate map}~$\g_e \colon V \to e(0)$. We also write $\{
\alpha(e), \omega(e) \}$ for~$e(0)$ and always make sure that the
choice of labels for these two vertices does not matter.  The vertex
set then splits as~$V = \g_e^{-1}(\alpha(e)) \sqcup
\g_e^{-1}(\omega(e))$.  Each of these sets is called a
\emph{(combinatorial) halfspace} (Figure~\ref{fig:halfspace}).  We
write~$\h$ for the set of halfspaces in $X$.  By definition, $\h$ is
equipped with a fixpoint-free involution~$\h \to \h \colon h \mapsto
\overline{h} \coloneqq V \setminus h$.  Another important property of
halfspaces is that they are \emph{convex}: Whenever $x, y \in h$, it
follows that $[x, y] \subset h$. This will be clear from
Lemma~\ref{lem:dist:halfspaces} below.

A \emph{hyperplane} is a set~$\{ h, \overline{h} \}$, where $h \in
\h$. If $h \in \h$ arises from~$\g_e$, we say that $e$ is \emph{dual}
to the halfspace~$h$ and to the hyperplane~$\{ h, \overline{h} \}$.
Every edge is dual to exactly one hyperplane, but different edges can
be dual to the same hyperplane if they induce the same partition
(Figure~\ref{subfig:dualedges}).

Let $x,y \in V$.  We say that a halfspace~$h$ \emph{separates~$y$
  from~$x$} if $y \in h$ and~$x \in \overline{h}$
(Figure~\ref{subfig:separates}); this notion is \emph{not} symmetric
in~$y$ and~$x$. Similarly, we may also speak of hyperplanes
separating~$x$ and~$y$.  A combinatorial geodesic~$\gamma = x_0 x_1
\cdots x_k$ is said to \emph{cross into} a halfspace~$h$ at time~$i$
if $h$ separates~$x_{i+1}$ from~$x_i$
(Figure~\ref{subfig:crosses}). In this case, we also say that $\gamma$
\emph{crosses} the hyperplane~$\{h,\overline h\}$. The number of times
$h$ is crossed into by~$\gamma$ refers to the number of times~$i$ such
that this holds. In fact, this number is always either~$0$ or~$1$.

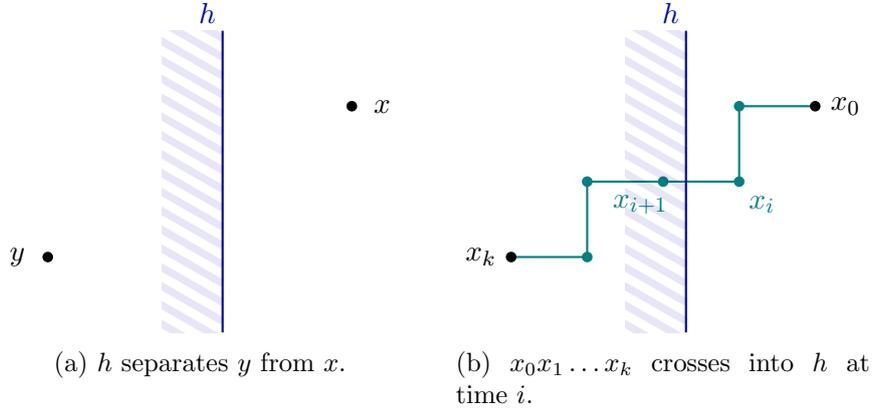
\begin{figure}
  \begin{center}
    \subcaptionbox{$h$ separates~$y$ from~$x$.%
      \label{subfig:separates}}{%
      \begin{tikzpicture}[x=1cm,y=1cm,thick]
        \vertx{(0,0)}
        \draw (-0.4,0) node {$y$};
        \vertx{(4,2)}
        \draw (4.4,2) node {$x$};
        \halfspace{(2.5,1)}{90}{h}
      \end{tikzpicture}}
  \hfil
  \subcaptionbox{$x_0x_1 \dots x_k$ crosses into~$h$ at time~$i$.%
    \label{subfig:crosses}}{%
    \begin{tikzpicture}[x=1cm,y=1cm,thick]
        \begin{scope}[\geodesiccolor]
          \draw (0,0) -- (1,0) -- (1,1) -- (3,1) -- (3,2) -- (4,2);
          \vertx{(1,0)}
          \vertx{(1,1)}
          \vertx{(2,1)}
          \vertx{(3,1)}
          \vertx{(3,2)}
          \draw (3.3,0.7) node {$x_i$};
          \draw (1.7,0.7) node {$x_{i+1}$};
        \end{scope}
        \vertx{(0,0)}
        \draw (-0.4,0) node {$x_k$};
        \vertx{(4,2)}
        \draw (4.4,2) node {$x_0$};
        \halfspace{(2.5,1)}{90}{h}
    \end{tikzpicture}}  
  \end{center}

  \caption{Separation by halfspaces.}
\end{figure}

\begin{lemma}[\cite{Sageev1995}{\cite[Remark~1.7]{Hagen}}]
  \label{lem:dist:halfspaces}
  In the situation of Setup~\ref{setup:cat0}, let $x,y \in V$. Then 
  the $\lone$-distance~$D(x, y)$ is equal to the number of
  hyperplanes that separate~$x$ and~$y$.  More precisely, for every
  hyperplane~$H$ separating~$x$ and~$y$, every geodesic from~$x$
  to~$y$ crosses~$H$ exactly once.
\end{lemma}

Because of Lemma~\ref{lem:dist:halfspaces}, the set of halfspaces
separating~$y$ from~$x$ will play an important role: We denote it
by~$[x, y]_{\h}$. Such sets are called \emph{$\h$-intervals}.  Notice
that it is possible that $[x, y]_{\h} = [x', y']_{\h}$ for distinct
pairs~$(x, y)$,~$(x', y')$.  The second part of
Lemma~\ref{lem:dist:halfspaces} implies that~$[x, y]_{\h} = [x,
  m]_{\h} \sqcup [m, y]_{\h}$ holds for every vertex~$m \in [x, y]$
belonging to an oriented geodesic segment between $x$ and~$y$. This
holds in particular when $m = m(x, y, z)$ for some third vertex~$z$.

\pagebreak

\subsubsection{Transversality}

\begin{defi}[transverse halfspaces]
\label{def:transversality_halfspaces}
  In the situation of Setup~\ref{setup:cat0}, 
  two halfspaces $h_1$ and~$h_2$ are \emph{transverse} if each of the
  four intersections $h_1 \cap h_2, \overline{h_1} \cap h_2, h_1 \cap
  \overline{h_2}$, and $\overline{h_1} \cap \overline{h_2}$ is
  non-empty (\cref{fig:transverse}).
   We then write~$h_1 \pitchfork h_2$ and also say that the pair $h_1, h_2$ is transverse.
\end{defi}

\begin{figure}
  \begin{center}
    \subcaptionbox{Transverse halfspaces.}{%
      \begin{tikzpicture}[x=1cm,y=1cm,thick]
        \begin{scope}[black!50]
          \squaregrid
        \end{scope}
        \halfspaces{(0,-0.5)}{180}{h_1}
        \def\halfspacecolor{\shalfspacecolor}
        \halfspaces{(0,0)}{90}{h_2}
    \end{tikzpicture}}
    \hfil
    \subcaptionbox{\raggedright Non-transverse half\-spaces.}{%
      \begin{tikzpicture}[x=1cm,y=1cm,thick]
        \begin{scope}[black!50]
          \squaregrid
        \end{scope}
        \halfspaces{(-1,0)}{90}{h_1}
        \def\halfspacecolor{\shalfspacecolor}
        \halfspaces{(1,0)}{90}{h_2}        
    \end{tikzpicture}}
  \end{center}

  \caption{(Non-)transverse halfspaces in the square grid.}
  \label{fig:transverse}
\end{figure}
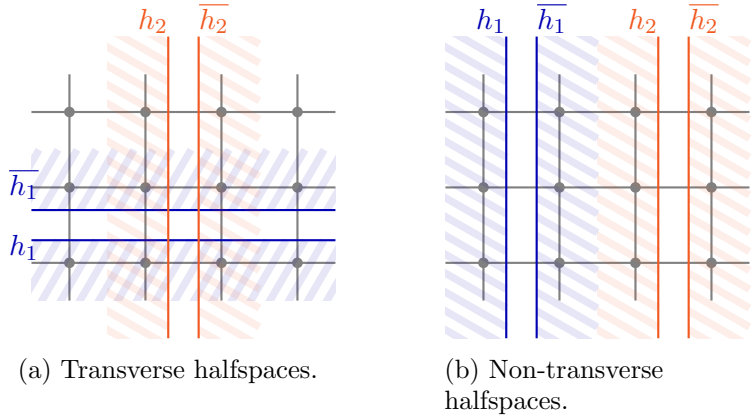

Whenever two halfspaces are transverse, this is witnessed in a quadrangle:

\begin{lemma}
\label{lem:quadrangle_from_transversality}
In the situation of Setup~\ref{setup:cat0}, let $h,k$ be halfspaces of~$X$ with~$h \pitchfork k$.
There is a quadrangle consisting of four edges~$e_{H}$, $f_{H}$,
    $e_{K}$, $f_{K}$ such that $e_{H}$, $f_{H}$ are dual to the
    hyperplane~$H\coloneqq \ls h, \overline{h} \rs$ and such that
    $e_{K}$, $f_{K}$ are dual to the hyperplane~$K\coloneqq \ls k,
    \overline{k} \rs$ (see \cref{fig:claim}).
\end{lemma}

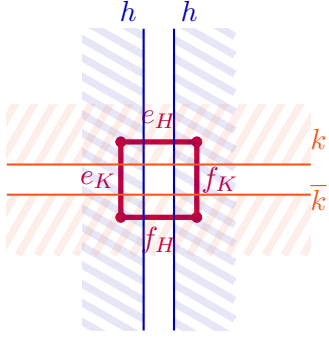
\begin{figure}
  \begin{center}
    \begin{tikzpicture}[x=1cm,y=1cm,thick]
      \begin{scope}[\hedgecolor]
        \hedge{(0,0)}{(1,0)}
        \hedge{(0,0)}{(0,1)}
        \hedge{(1,0)}{(1,1)}
        \hedge{(0,1)}{(1,1)}
        \draw {(0.5,1.3)} node {$e_H$};
        \draw {(0.5,-0.3)} node {$f_H$};
        \draw {(-0.3,0.5)} node {$e_K$};
        \draw {(1.3,0.5)} node {$f_K$};
      \end{scope}
      \halfspaces{(0.5,0.5)}{90}{h}
      \def\halfspacecolor{\shalfspacecolor}
      \halfspaces{(0.5,0.5)}{0}{k}
    \end{tikzpicture}
  \end{center}

  \caption{The quadrangle in Lemma~\ref{lem:quadrangle_from_transversality}}
  \label{fig:claim}
\end{figure}

Given a hyperplane $H = \{h, \overline{h}\}$, we denote by $\mathcal{N}(H)$ the \emph{carrier} of $H$,
namely the convex subcomplex of $X$ spanned by all edges that cross $H$.
We will use the following property of carriers:

\begin{lemma}
\label{lem:carrier}

Let $H$ be a hyperplane and let $\mathcal{N}(H)$ be its carrier.
Let $\gamma$ be a geodesic contained in $\mathcal{N}(H)$.
Then every hyperplane crossed by $\gamma$ is either equal to $H$ or transverse to $H$.

In particular, if both endpoints of $\gamma$ are contained in one side of $H$,
then every hyperplane crossed by $\gamma$ is transverse to $H$.
\end{lemma}

\subsubsection{Finiteness conditions}

We will be working with \emph{finite-dimensional} $\cat0$-cube
complexes, where the dimension of~$X$ is the highest dimension of a
cube in~$X$.

Moreover, we will impose the condition that the complexes have
finite staircase length:

\begin{defi}[{\cite[Definition 4.14]{elia}}]
  Let $X$ be a $\cat0$ cube complex, let $\sigma \in \N$.  A
  \emph{length-$\sigma$ staircase} in $X$ is a pair~$(h_1 \supset
  \cdots \supset h_\sigma, k_1 \supset \cdots \supset k_\sigma)$
  of proper chains of halfspaces with the following properties
  (Figure~\ref{fig:staircase}):
  \begin{itemize}
  \item For all~$i \in \{1,\dots, \sigma\}$ and all~$j \in \{1,\dots,i-1\}$,
     we have~$h_i \pitchfork k_j$.
  \item For all~$i \in \{1,\dots, \sigma\}$,  we have~$h_i \supsetneq k_i$. 
  \end{itemize}
  The \emph{staircase length} of~$X$ is the maximal length of a
  staircase in~$X$ (or~$\infty$).
\end{defi}

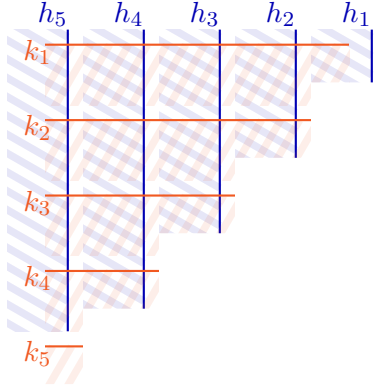
\begin{figure}
  \begin{center}
    \begin{tikzpicture}[x=1cm,y=1cm,thick]
      \clip (5,2.5) -- (5,1.3)
      -- (4,1.3) -- (4,0.3)
      -- (3,0.3) -- (3,-0.7)
      -- (2,-0.7) -- (2,-1.7)
      -- (1,-1.7) -- (1, -2.7)
      -- (1,-2.7) -- (0,-2.7)
      -- (0,2.5) -- cycle;
      \clip (0,-3) rectangle +(5,6);
      \foreach \i in {1,...,5} {%
        \halfspace{(6-\i,0)}{90}{h_{\i}}
      }
      \begin{scope}[shift={(2.5,3)}]
        \def\halfspacecolor{\shalfspacecolor}
        \foreach \j in {1,...,5} {%
          \halfspace{(0,-\j)}{180}{k_{\j}}
        }
      \end{scope}
    \end{tikzpicture}
  \end{center}
  
  \caption{A staircase, schematically}
  \label{fig:staircase}
\end{figure}

\begin{rem}
  The numbering differs slightly from Fioravanti's convention~\cite{elia},
  since this will make our arguments more transparent.
\end{rem}

\begin{example}
  \hfil
  \begin{itemize}
  \item Every tree has staircase length at most~$1$.
  \item Salvetti complexes have finite staircase length (Section~\ref{s:raag}).
  \item There exist finite-dimensional $\cat0$ cube complexes
    with infinite staircase length, e.g., the infinite staircase
    in Figure~\ref{fig:infstaircase}.
  \end{itemize}
\end{example}

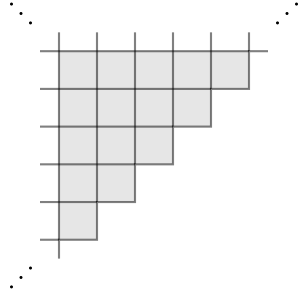
\begin{figure}
  \begin{center}
    \begin{tikzpicture}[x=0.5cm,y=0.5cm,thick]
      \begin{scope}
      \fill[opacity=0.1] (0,0) -- (0,5)
      -- (5,5) -- (5,4) -- (4,4) -- (4,3) -- (3,3) -- (3,2) -- (2,2) -- (2,1) -- (1,1) -- (1,0)
      -- cycle;
      \draw[opacity=0.5] (5.5,5) -- (5,5) -- (5,4) -- (4,4) -- (4,3) -- (3,3) -- (3,2) -- (2,2) -- (2,1) -- (1,1) -- (1,0)
      -- (0,0) -- (0,-0.5) ;
      \foreach \j in {0,...,5} {
        \draw[opacity=0.5] (\j,5.5) -- (\j,\j);
        \draw[opacity=0.5] (-0.5,\j) -- (\j,\j);
      }
      \def\svertx#1{%
        \fill #1 circle (0.02cm);}
      \def\gddots#1{\begin{scope}[shift={#1}]
          \svertx{(-0.25,0.25)}
          \svertx{(0,0)}
          \svertx{(0.25,-0.25)}
      \end{scope}}
      \def\giddots#1{\begin{scope}[shift={#1}]
          \svertx{(-0.25,-0.25)}
          \svertx{(0,0)}
          \svertx{(0.25,0.25)}
      \end{scope}}
      \gddots{(-1,6)}
      \giddots{(6,6)}
      \giddots{(-1,-1)}
      \end{scope}
    \end{tikzpicture}
  \end{center}
  
  \caption{Beware of the infinite staircase!}
  \label{fig:infstaircase}
\end{figure}

\subsection{Median quasimorphisms and median classes}

Median quasimorphisms are a generalisation of counting quasimorphisms
(Example~\ref{exa:brooks}).
Instead of subwords in free groups, we count segments of crossed
halfspaces for groups acting on $\cat0$ cube complexes. We start with
the definition of segments in terms of halfspaces~\cite{FFT}.

\begin{defi}[tightly nested, segment]
  In the situation of Setup~\ref{setup:cat0}, 
  two halfspaces $h_1$ and~$h_2$ are \emph{tightly nested} if they are
  distinct, if $h_1 \supset h_2$, and if there is no other halfspace~$h$ such
  that~$h_1 \supset h \supset h_2$ (Figure~\ref{subfig:tight}).

  An \emph{${\h}$-segment} of length~$l \in \N$ is a sequence~$(h_1
  \supset \cdots \supset h_l)$ of tightly nested halfspaces
  (Figure~\ref{subfig:segment}).  The \emph{reverse} of~$s = ( h_1
  \supset \cdots \supset h_l )$ is~$\overline{s} \coloneqq (
  \overline{h_l} \supset \cdots \supset \overline{h_1} )$.  We denote
  by~$X_{\h}^{(l)}$ the set of all $\h$-segments of length~$l$ in~$X$.
  Given~$x, y \in V$, we write~$[x, y]_{\h}^{(l)}$ for the set of all
  segments of length~$l$ that are contained in~$[x, y]_{\h}$.

  A vertex~$x$ is said to be in the \emph{interior} of the segment~$s
  = ( h_1 \supset \cdots \supset h_l )$ if~$x \in h_1 \cap
  \overline{h_l}$ (Figure~\ref{subfig:segment}).

  The action of~$\Gamma$ on~$X$ induces an action of~$\Gamma$
  on~$X_{\h}^{(l)}$; the orbit of~$s \in X_{\h}^{(l)}$ is denoted
  by~$\Gamma s$, we call its elements \emph{translates} of~$s$.
\end{defi}

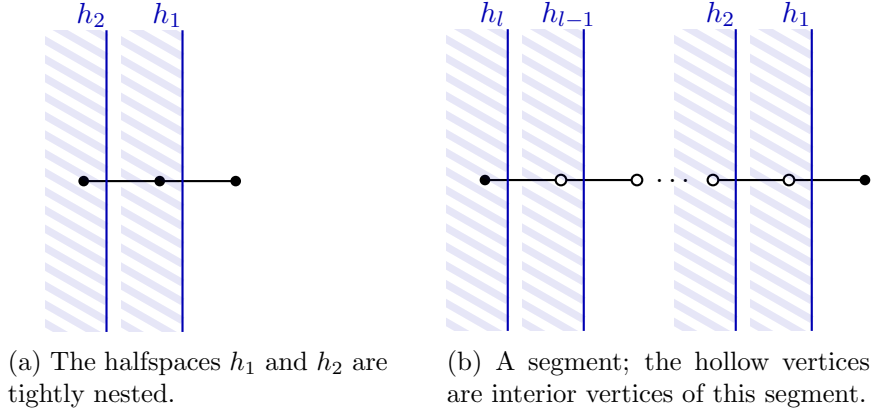
\begin{figure}
  \begin{center}
    \subcaptionbox{The half\-spaces~$h_1$ and~$h_2$ are tightly nested.%
      \label{subfig:tight}}{%
      \begin{tikzpicture}[x=1cm,y=1cm,thick]
        \clip (-1,-2) rectangle +(5,4.5);
        \draw (0,0) -- (2,0);
        \vertx{(0,0)}
        \vertx{(1,0)}
        \vertx{(2,0)}
        \halfspace{(0.5,0)}{90}{h_2}
        \halfspace{(1.5,0)}{90}{h_1}
    \end{tikzpicture}}
    \hfil
    \subcaptionbox{A segment; the hollow vertices are interior
      vertices of this segment.\label{subfig:segment}}{%
      \begin{tikzpicture}[x=1cm,y=1cm,thick]
        \draw (0,0) -- (2,0);
        \draw (2.5,0) node {$\dots$};
        \draw (3,0) -- (5,0);
        \vertx {(0,0)}
        \vertx {(5,0)}
        \hollowvertx{(1,0)}
        \hollowvertx{(2,0)}
        \hollowvertx{(3,0)}
        \hollowvertx{(4,0)}
        \halfspace{(0.5,0)}{90}{h_l}
        \halfspace{(1.5,0)}{90}{h_{l-1}}
        \halfspace{(3.5,0)}{90}{h_2}
        \halfspace{(4.5,0)}{90}{h_1}
    \end{tikzpicture}}
  \end{center}

  \caption{Tightly nested halfspaces and a segment.}
\end{figure}

In analogy with counting quasimorphisms, we introduce the
associated counting functions, which will then lead to median
quasimorphisms and classes.

\begin{defi}[median quasimorphism]\label{def:medianqc}
  In the situation of Setup~\ref{setup:cat0}, let $l \in \N$, and 
  let $s \in X_{\h}^{(l)}$.
  We define the \emph{median quasimorphism~$f_s \colon V \times V \to \R$ for~$s$}
  for all~$(x,y) \in V \times V$ to be the number of translates of~$s$
  in~$[x,y]_{\h}$ minus the number of translates of~$s$ in~$[y,x]_{\h}$.
\end{defi}

By construction, $f_s = f_{\gamma s}$ holds for all~$\gamma \in \Gamma$.
Moreover, the cochain~$f_s$ is $\Gamma$-invariant. 

\begin{rem}
\label{rem:inversion}
  Note that whenever $s$ and $\bar{s}$ are in the same $\Gamma$-orbit,
  it follows from the definition that~$f_s = 0$.  If this is not the
  case, then the following function~$\es\colon X_{\h}^{(l)} \to \ls -1,0,1 \rs$
  is well-defined:
  \begin{equation*}
	\es(t) = \begin{cases}
	1, & \text{if } \Gamma t = \Gamma s,  \\
	-1, & \text{if } \Gamma t = \Gamma \bar{s}, \\
	0, & \text{otherwise}
	\end{cases}
  \end{equation*}
  Then, $f_s$ is given by
  \begin{align}\label{formula:fs}
  \begin{split}
    f_s \colon V^2 & \to \R \\
    (x,y) & \longmapsto 
    \sum_{t\in [x,y]_{\h}^{(l)}} \es(t).
  \end{split}
  \end{align}
  We will use this formula for all of our computations.
\end{rem}

It is not immediately clear that $f_s$ indeed is a quasimorphism
of~$\Gamma \actson V$, i.e., that $\delta^1 f_s$ is bounded. Indeed,
this is not true in full generality:

\begin{example}\label{exa:dunbounded_dim}
  For~$n \geq 1$, let $X_n^- \coloneqq [-1, 0]^n$, $X_n^+ \coloneqq
  [0, 1]^n$, and let $X_n$ be the wedge~$X_n^- \vee X_n^+$, glued
  along the point~$0 \in X_n^- \cap X_n^+$.  Then $X_n$ can be viewed
  as a $\cat0$ cube complex, and together with the canonical
  inclusions~$X_n \hookrightarrow X_{n+1}$, we obtain a directed
  system whose direct limit~$X = X^+ \vee X^-$ is an
  infinite-dimensional $\cat0$ cube complex
  (Figure~\ref{fig:dunbounded}).  Moreover, $X$ admits an action of
  the group~$\Gamma \coloneqq S_\infty \times S_\infty$, where
  $S_\infty$ denotes the group of permutations of a countably infinite
  set with finite support.  Here the factors act on $X^-$ and $X^+$
  respectively by permuting the coordinates; in particular, $0$ is
  fixed by both factors.  Let $h^-$ be the halfspace $\{ (0, *) \in
  X^- \} \cup X^+$, and $h^+$ the halfspace $\{ (1, *) \in X^+
  \}$. Then $h^- \supset h^+$ are tightly nested, so they define a
  segment~$s$.  Let further $x_n^- \coloneqq (-1, \ldots, -1) \in
  X_n^- \subset X^-$ and $x_n^+ \coloneqq (1, \ldots, 1) \in X_n^+
  \subset X^+$. Then $s \in [x_n^-, x_n^+]_{\h}^{(2)}$.

\begin{figure}
  \begin{center}
    \def\cube{%
      \draw (0,0) rectangle +(1,1);
      \draw (0.5,0.5) rectangle +(1,1);
      \foreach \x in {(0,0),(1,0),(0,1),(1,1)} {%
        \draw \x -- +(0.5,0.5);
      }
    }
    \begin{tikzpicture}[x=1cm,y=1cm,thick]
      \cube
      \vertx{(1.5,1.5)}
      \draw (1.9,1.5) node {$x_n^+$};
      \draw (-0.5,1) node {$X^+$};
      \begin{scope}[shift={(-1.5,-1.5)}]
        \cube
        \vertx{(0,0)}
        \draw (-0.4,0) node {$x_n^-$};
        \draw (2,0.5) node {$X^-$};
      \end{scope}
      \vertx{(0,0)}
      \halfspace{(-0.75,-1)}{0}{\quad h^-}
      \halfspace{(0.75,0.5)}{0}{\quad h^+}
    \end{tikzpicture}
  \end{center}

  \caption{The construction of Example~\ref{exa:dunbounded_dim}, schematically.}
  \label{fig:dunbounded}
\end{figure}
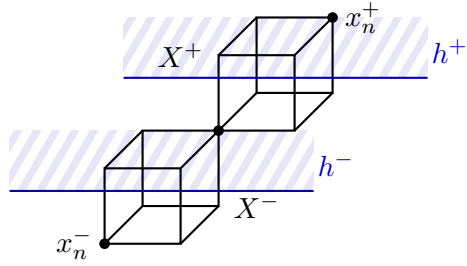

  Since the action of~$\Gamma$ fixes~$0$ and preserves $X^-$
  and~$X^+$, the orbit~$\Gamma s$ does not contain~$\bar{s}$. Moreover, no
  translate of~$\bar{s}$ belongs to~$[x_n^-, x_n^+]_{\h}^{(2)}$.  On
  the other hand, $[x_n^-, x_n^+]_{\h}^{(2)}$ contains many translates
  of~$s$: For instance, every element of~$S_n \times S_n \leq \Gamma$
  fixes~$x_n^-$, $0$ and~$x_n^+$, so it sends~$s$ to a segment
  in~$[x_n^-, x_n^+]_{\h}^{(2)}$.  But $h^-$ and $h^+$ are defined by
  conditions on the first coordinate only, so choosing elements
  in~$S_n$ that do not fix this coordinate, we obtain $n^2$~distinct
  translates of~$s$ that still lie in~$[x_n^-, x_n^+]_{\h}^{(2)}$.
  Finally, since~$\Gamma$ preserves $X^-$ and $X^+$, there can be
  no translate of~$s$ or~$\bar{s}$ inside~$[x_n^-, 0]^{(2)}_{\h}$ or~$[0, x_n^+]_{\h}^{(2)}$.

  The above discussion shows that
  \begin{align*}
  \delta^1 f_s(x_n^-, 0, x_n^+)
  &= f_s(0, x_n^+) - f_s(x_n^-, x_n^+) + f_s(x_n^-, 0) \\
  &= 0 - \bigl| \Gamma s \cap [x_n^-, x_n^+]_{\h}^{(2)} \bigr| + 0 \leq -n^2.
  \end{align*}
  Varying~$n$, we deduce that $\delta^1 f_s$ is unbounded.
\end{example}

\begin{rem}
  The issue in \cref{exa:dunbounded_dim} is not that $\Gamma$ is
  infinitely generated. The same argument can be run with the direct
  square of any group of permutations of a countably infinite set that
  contains~$S_\infty$.  For instance, we could choose the Houghton
  group~$H_k$~\cite{Houghton}, which is of type~$F_{k-1}$.
\end{rem}

The following examples shows that also for finite dimensional complexes, $f_s$ need not be a quasimorphism. The obstruction here is given by infinite staircase length. (Note that the example \cref{exa:dunbounded_dim} has finite staircase length, simply because there are no proper chains of halfspaces of length greater than 2.)

\begin{example}\label{exa:dunbounded_staircase}
	Let $X$ be the infinite staircase depicted in Figure~\ref{fig:infstaircase}.
	To describe points in $X$, we interpret it as being embedded in $\R^2$, with vertices lying on $\Z^2$ and the ``right-most points'' given by the diagonal of $\Z^2$, that is $V = \{(x,y)\in \Z^2 \mid x\leq y\}$.
	Let $\Gamma \coloneqq \Z$ act on $X$ by shifting along the diagonal of $\Z^2$, i.e., $\gamma\cdot (x,y) = (x+\gamma, y+\gamma)$ for $\gamma\in \Z$, $(x,y)\in V$.
	
	We define $h^x = \ls (x,y)\in V \mid y > 0 \rs$ and $h^y = \ls (x,y)\in V \mid x > 0 \rs$ to be the halfspaces that are dual to the edges $((0,0), (0,1))$ and $((0,1),(1,1))$, respectively, and point to the ``upper right''. These halfspaces are tightly nested $h^x \supset h^y$, so define a segment $s$.
	As the action of $\Gamma$ preserves the orientation of $X\subset \Z^2$, the orbit~$\Gamma s$ does not contain~$\bar{s}$. We define a sequence of triples in $V$ that witnesses that $\delta^1 f_s$ is unbounded: For $n \in \N$, let $x_n = (0,0)$, $y_n = (0,n+1)$, $z_n = (n+1,n+1)$.
	Clearly, $\gamma h^x = \ls (x,y)\in V \mid y > \gamma \rs$ contains $y_n$ if and only if it contains $z_n$ and $\gamma h^y = \ls (x,y)\in V \mid x > \gamma \rs$ contains $x_n$ if and only if it contains $y_n$. The same is true for $\bar{h^x}$ and $\bar{h^y}$. This implies that $f_s(y_n, z_n) = 0 = f_s(x_n, y_n)$.
	Similarly, it is not hard to see that $[x_n, z_n]_{\h}^{(2)}$ contains $n$ translates of $s$ and no translate of $\bar{s}$. Hence, we get
  \begin{align*}
  \delta^1 f_s(x_n, y_n, z_n)
  &= f_s(y_n, z_n) - f_s(x_n, z_n) + f_s(x_n, y_n) \\
  &= 0 - n + 0.
  \end{align*}
  Varying~$n$, we deduce that $\delta^1 f_s$ is unbounded.
\end{example}

A way to avoid both pathological examples above is to add the hypotheses
of finite-dimensionality and finite staircase length. It will turn out that imposing these hypotheses is actually sufficient for showing that $f_s$ is a quasimorphism.
As a
preparation, we show the following estimate, which will also be crucial in the proof of Theorem~\ref{thm:mainccc}:

\begin{lemma}
  \label{lem:fd:segments:fix}
  Let $X$ be a $\cat0$ cube complex of finite dimension~$d$ and with finite
  staircase length~$\sigma$. Then for all~$x, y \in V$ and all~$m \in
  [x, y]$, the number of segments in~$[x,y]_{\h}^{(l)}$ that
  contain~$m$ in their interior is at most~$(l-1) \sigma d^l$.
\end{lemma}

\begin{proof}
  We start by showing the lemma in the case~$l = 2$.  The statement
  gives the bound~$\sigma d^2$ on the number of pairs of tightly
  nested halfspaces~$h \supset k$ such that $x \in
  \overline{h}$, $m \in h \cap \overline{k}$, and $y \in k$.

  Let $B$ be the set of all such pairs, and \emph{assume}
  by contradiction that~$|B| > \sigma d^2$.
  The set of all halfspaces separating~$x$ from~$y$ decomposes
  into a disjoint union of $d$ (possibly empty)
  chains~\cite[Lemma~1.16]{propertyA}. Hence, each pair in~$B$
  belongs to one of~$d^2$ types of chain-memberships of the
  two components. 
  By the pigeonhole principle, there exist $\sigma +1$~distinct
  pairs in~$B$ such that the first components form a
  chain and simultaneously also the second components form
  a chain.
  We number the pairs~$(h_1,
  k_1), \ldots, (h_{\sigma+1}, k_{\sigma+1})$ in such a way that $h_1
  \supset h_2 \supset \cdots \supset h_{\sigma} \supset h_{\sigma+1}$.
  

  First, we claim that $k_i \supset k_{i+1}$, for all~$i \in
  \{1,\dots,\sigma\}$. As the second components form a
  chain, we have $k_i \supset k_{i+1}$ or~$k_{i+1} \supset k_i$.
  However, in the latter case, we would obtain the
  chain~$h_i \supset h_{i+1} \supset k_{i+1} \supset k_i$.
  Because $(h_i, k_i)$ is a tightly nested pair and $h_{i+1} \neq k_{i+1}$,
  this would imply~$h_i = h_{i+1}$ and $k_{i + 1} = k_i$, which
  would contradict that the pairs~$(h_i,k_i)$ and $(h_{i+1}, k_{i+1})$
  are different. Therefore, $k_i \supset k_{i+1}$.
  
  We obtained the inclusions $h_1 \supset \cdots \supset h_{\sigma+1}$
  and $k_1 \supset \cdots \supset k_{\sigma+1}$.  Next we show that
  they are all proper:  \emph{Assume} for a contradiction
  that there is an~$i \in \{1,\dots, \sigma\}$ with~$h_i = h_{i+1}$
  (the case~$k_i = k_{i+1}$ is symmetric).
  Then $k_i \neq k_{i+1}$, since the pairs~$(h_i,k_i)$ and
  $(h_{i+1}, k_{i+1})$ are different. Thus, $h_{i+1} = h_i
  \supsetneq k_i \supsetneq k_{i+1}$, which contradicts that $h_{i+1}
  \supset k_{i+1}$ are tightly nested.

  Finally, we show that $h_i \pitchfork k_j$ for all~$i \in \{1,\dots,
  \sigma+1\}$ and all~$j \in \{1,\dots, i-1\}$.  This will show that
  $h_1 \supset \cdots \supset h_{\sigma+1}$ and $k_1 \supset \cdots
  \supset k_{\sigma+1}$ form a staircase of length~$\sigma+1$,
  contradicting the hypothesis that $X$ has staircase length~$\sigma$.
  So, let $i \in \{1,\dots, \sigma+1\}$ and $j \in \{1,\dots, i-1\}$.
  Then $\overline {h_i} \cap \overline {k_j} \neq \emptyset$ (because
  of~$x$) and $h_i \cap \overline k_j \neq \emptyset$ (because of~$m$)
  as well as $h_i \cap k_j \neq \emptyset$ (because of~$y$). It
  remains to show that $\overline h_i \cap k_j$ is non-empty. If
  $\overline h_i \cap k_j$ were empty, then we would have~$h_i \supset
  k_j$, and thus $h_j \supsetneq h_i \supsetneq k_j$, which
  contradicts that $h_j \supsetneq k_j$ are tightly nested. Hence,
  $\overline h_i \cap k_j \neq \emptyset$.  This concludes the proof
  in the case~$l=2$.

  Now we prove the general case.
  Consider a segment $t = (h_1 \supset \cdots
  \supset h_{l}) \in [x, y]_{\h}^{(l)}$. Since $m \in [x, y]$,
  there exists~$i\in \{1,\ldots, l-1\}$ such that $m \in h_i \cap \overline{h_{i+1}}$, and the previous case gives at most $\sigma d^2$ options for the pair $(h_i, h_{i+1})$.
  We claim that in addition to the $l-1$ options for $i$, there are at most $d$ options for each of the remaining $(l-2)$ halfspaces.
  Indeed, assume that we have already chosen $h_{i+2}, \ldots, h_j$ and we need to choose $h_{j+1}$.
  The latter needs to lie in $[\overline{h_j}, y]_{\h}$, that is, the set of halfspaces containing $y$ and contained in $h_j$.
  Moreover, it needs to be a maximal element of this set, because $(h_j, h_{j+1})$ must be tightly nested.
  But maximal elements are pairwise transverse, and $X$ is $d$-dimensional, so there are at most $d$ options for $h_{j+1}$.
  The same argument, assuming that $h_j, \ldots, h_{i-1}$ have been chosen, gives at most $d$ options for $h_{j-1}$, and concludes the proof.
\end{proof}

\begin{prop}[defect estimate]
  \label{prop:fs:qm}
  In the situation of Setup~\ref{setup:cat0}, let $X$ be a
  $\cat0$-cube complex of dimension~$d$
  and finite staircase length~$\sigma$.
  Let $l \in \N$ and
  let $s \in X_{\h}^{(l)}$ be a segment. Then $\delta^1 f_s \colon V^3
  \to \R$ is bounded.  More precisely,
  \[ \| \delta^1 f_s \|_\infty \leq 3(l-1) \sigma d^l.
  \]
\end{prop}

\begin{proof}
  The statement is obvious if $f_s = 0$, so we may assume that $\Gamma
  s \neq \Gamma \bar{s}$ and use Equation~(\ref{formula:fs}) from
  Remark~\ref{rem:inversion}.
  As a first step, we consider the following situation: 
  Let $x, y \in V$, and let $m \in [x, y]$. Then $[x, y]_{\h} =
  [x, m]_{\h} \sqcup [m, y]_{\h}$, and thus
  \begin{align*}
    f_s(x, y) &= \sum\limits_{t \in [x, y]_{\h}^{(l)}} \es(t) \\
    &= \sum\limits_{t \in [x, m]_{\h}^{(l)}} \es(t) + \sum\limits_{t \in [m, y]_{\h}^{(l)}} \es(t)
    + \sum\limits_{t \in [x, y]_{\h}^{(l)} \setminus [x, m]_{\h}^{(l)} \cup [m, y]_{\h}^{(l)}} \es(t) \\
    &= f_s(x, m) + f_s(m, y) + \sum\limits_{t \in [x, y]_{\h}^{(l)} \setminus [x, m]_{\h}^{(l)} \cup [m, y]_{\h}^{(l)}} \es(t).
  \end{align*}
  A segment~$t$ appears in the last sum only if $m$ belongs to the
  interior of~$t$ (Figure~\ref{fig:fs:qm}).  Therefore,
  Lemma~\ref{lem:fd:segments:fix} shows that there are at
  most~$(l-1) \sigma d^l$
  terms in this sum.

  For the general case, let $x, y, z \in V$ and $m \coloneqq m(x, y, z)$. Then
  the previous estimate shows that 
  \begin{align*}
    \bigl| \delta^1 f_s(x, y, z) \bigr|
    &= \bigl| f_s(y, z) - f_s(x, z) + f_s(x, y) \bigr| \\
    &\leq 3 \cdot (l-1) \sigma d^l
     + A ,
  \end{align*}
  where the term
  \[ A := \bigl|(f_s(y, m) + f_s(m, z)) - (f_s(x, m) + f_s(m, z)) + (f_s(x, m) + f_s(m, y))\bigr| 
  \]
  vanishes since $f_s$ is antisymmetric.
\end{proof}

\begin{figure}
  \begin{center}
    \subcaptionbox{}{%
      \begin{tikzpicture}[x=1cm,y=1cm,thick]
      \vertx{(0,0)}
      \draw {(-0.3,0)} node {$y$};
      \vertx{(5,2)}
      \draw {(5.3,2)} node {$x$};
      \draw (0,0) -- (1,0) -- (1,1) -- (4,1) -- (4,2) -- (5,2);
      \begin{scope}[color=\hedgecolor]
        \vertx{(2,1)}
        \draw (2,0.7) node {$m$};
        \fill[opacity=0.1] (2,1) circle (0.9);
      \end{scope}
    \end{tikzpicture}}
    \hfil
    \subcaptionbox{}{%
      \begin{tikzpicture}[x=1cm,y=1cm,thick]
      \vertx{(0,0)}
      \draw {(-0.3,0)} node {$y$};
      \vertx{(5,2)}
      \draw {(5.3,2)} node {$x$};
      \draw (0,0) -- (1,0) -- (1,1) -- (4,1) -- (4,2) -- (5,2);
      \vertx{(2,2)}
      \draw (2,1) -- (2,2);
      \draw (2.3,2) node {$z$};
      \begin{scope}[color=\hedgecolor]
        \vertx{(2,1)}
        \draw (2,0.7) node {$m$};
        \fill[opacity=0.1] (2,1) circle (0.9);
      \end{scope}      
    \end{tikzpicture}}
  \end{center}

  \caption{The critical areas in the two cases of the proof of \cref{prop:fs:qm}.}
  \label{fig:fs:qm}
\end{figure}
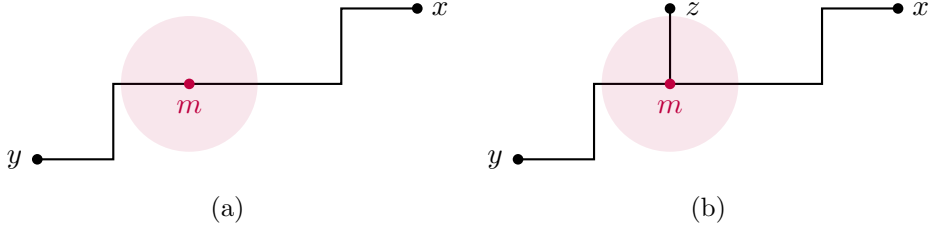

It is instructive to consider the case of a (big) Brooks quasimorphism $H_w$.
Thn $d = \sigma = 1$, $l = |w|$ and the $\ell^\infty$-norm of the coboundary coincides with the defect.
Hence, the bound $3(|w|-1)$ from Proposition~\ref{prop:fs:qm} matches the bound from \cref{ss:qm}.

\begin{defi}[median class]
  In the situation of Proposition~\ref{prop:fs:qm},
  by Remark \ref{rem:qmbc}),
  we obtain a
  bounded cohomology class~$[\delta^1f_s] \in \HH_{\Gamma,b}^2(X;\R)$.
  We call this the \emph{median class} of~$s$.
\end{defi}

\begin{rem}\label{rem:ueber}
  Related bounded cohomology classes appeared in the work of Chatterji,
  Fern\'os, and Iozzi~\cite{CFI}: Note that in that case too, the
  boundedness of the relevant classes is achieved by working in the
  finite-dimensional setting and additional restrictions on the segments.
  Similarly to these restrictions on the segments, we could also
  replace the hypothesis on finite staircase length by imposing more
  refined constraints on the segments than consecutive halfspaces
  being ``tightly nested''.
  
  Also, non-overlapping versions of the quasimorphisms by Chatterji,
  Fern\'os, and Iozzi were considered~\cite{FFT,FFTsimp}; in that
  situation, the finite-dimensionality hypothesis is not needed, and
  the defect is uniformly bounded.
\end{rem}

%

\subsection{Cup products of median classes with non-transverse classes}
\label{ss:cupccc}

In general, cup products of median classes can be non-trivial:

\begin{example}
  \label{ex:FxF}
  Let $F$ be a non-abelian free group and let $T$ be a Cayley tree
  of~$F$.  We consider the product action~$F \times F \actson X \coloneqq T
  \times T$ of the translation action~$F \actson T$. Then $X$ carries
  a canonical structure of a $\cat0$ cube complex.
  
  Let $\varphi \in \HH^2_b(F;\R) \setminus \{0\}$. For the canonical
  projections~$p_1, p_1 \colon F \times F \to F$, we then
  obtain~\cite[Proposition~3.3]{clara}
  \[ \HH^2_b(p_1;\R)(\varphi) \cup \HH^2_b(p_2;\R)(\varphi) \neq 0
  \]
  in~$\HH^4_{b}(F \times F;\R) \cong \HH^4_{F \times F,b}(X;\R)$. 
  Moreover, if $\varphi$ is induced by the Brooks
  quasimorphism of~$w \in F$, then $\HH^2_b(p_i;\R)(\varphi)$
  are the classes given by the median quasimorphisms of the segments
  induced by~$w$ on~$X$ via the two factors~$T$ in~$X$.

  More generally, let $X \coloneqq T_1 \times T_2$ be a product of two
  trees. Every edge~$e$ in~$X$ is of the form $e_1 \times x_2$, where
  $e_1$ is an edge in $T_1$ and $x_2$ is a vertex in $T_2$, or the
  other way around.  If $h_1, \overline{h_1}$ are the halfspaces dual
  to~$e_1$ in~$T_1$, then $h_1 \times T_2$ and $\overline{h_1} \times
  T_2$ are the halfspaces dual to~$e$ in~$X$.  Therefore an
  $\h$-segment~$s_i$ in~$T_i$ defines a segment~$\widehat{s_i}$
  in~$X$.

  Now suppose that $\Gamma$ acts on each $T_i$, and equip $X$ with the
  product action of~$\Gamma \times \Gamma$.  Then $[\delta^1
    f_{\widehat{s_i}}] \in \HH^2_{\Gamma, b}(X; \R)$ is the pullback
  of~$[\delta^1 f_{s_i}] \in \HH^2_{\Gamma, b}(T_i; \R)$ under the
  projection~$X \to T_i$. The same arguments as for the free group case 
  show: If both $[\delta^1 f_{s_i}]$ are non-trivial, then
  \[ [\delta^1 f_{\widehat{s_1}}] \cup [\delta^1 f_{\widehat{s_2}}] \neq 0
  \in \HH^4_{\Gamma\times \Gamma,b}(X;\R),
  \]
  so the cup product is non-trivial.
\end{example}

We will prove the vanishing of cup products under an additional
assumption that avoids the above behaviour.

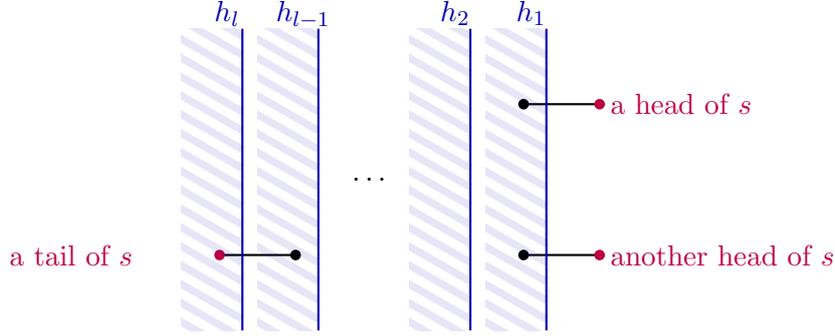
\begin{figure}
  \begin{center}
    \begin{tikzpicture}[x=1cm,y=1cm,thick]
      \draw (0,-1) -- +(1,0);
      \vertx{(1,-1)}
      \draw (5,-1) -- +(-1,0);
      \draw (5,1) -- +(-1,0);
      \vertx{(4,1)}
      \vertx{(4,-1)}
      \draw (2,0) node {$\dots$};
      \begin{scope}[\hedgecolor]
        \vertx{(0,-1)}
        \draw {(-1,-1)} node[anchor=east] {a tail of~$s$};
        \vertx{(5,-1)}
        \vertx{(5,1)}
        \draw {(5,-1)} node[anchor=west] {another head of~$s$};
        \draw {(5,1)} node[anchor=west] {a head of~$s$};
      \end{scope}
      \halfspace{(4.5,0)}{90}{h_1}
      \halfspace{(3.5,0)}{90}{h_2}
      \halfspace{(1.5,0)}{90}{h_{l-1}}
      \halfspace{(0.5,0)}{90}{h_l}
    \end{tikzpicture}
  \end{center}

  \caption{Heads and tails of segments.}
  \label{fig:headstails}
\end{figure}

\begin{defi}[heads/tails]
\label{defi:headstails}
  In the situation of Setup~\ref{setup:cat0}, 
  let $s = ( h_1 \supset \cdots \supset h_l ) \in X_{\h}^{(l)}$.  We
  say that~$\alpha \in V$ is a \emph{head} of~$s$ if $\alpha
  \in \overline{h_1}$ and there exists an
  edge dual to~$h_1$ that has $\alpha$ as one of its endpoints
  (Figure~\ref{fig:headstails}).
  We say that~$\omega \in V$ is a \emph{tail}
  of~$s$ if~$\omega \in h_l$ and there exists an edge dual to~$h_l$ 
  that has $\omega$ as one of its endpoints.

  We let $\alpha(s)$ denote the set of heads of~$s$ and we let
  $\omega(s)$ denote the set of tails of~$s$. By definition,
  $\alpha(\overline{s}) = \omega(s)$ and $\omega(\overline{s}) =
  \alpha(s)$.
\end{defi}

\begin{defi}[non-transverse]
  \label{defi:nontransverse}
  In the situation of Setup~\ref{setup:cat0}, 
  let $s \in X_{\h}^{(l)}$ and let $\kappa \in \CC_{\Gamma, b}^n(X,
  \R)$.  We say that $\kappa$ and $s$ are \emph{non-transverse} if for
  all~$x_1, \ldots, x_n \in V$, the value of~$\kappa(\alpha, x_1,
  \ldots, x_n)$ is constant over all~$\alpha \in \alpha(s)$, and the
  value of~$\kappa(\omega, x_1, \ldots, x_n)$ is constant over
  all~$\omega \in \omega(s)$.  Given a set~$S \subset X_{\h}^{(l)}$, we
  say that $\kappa$ and~$S$ are \emph{non-transverse} if $\kappa$
  and~$s$ are non-transverse for all~$s \in S$.

  We define similarly non-transversality with respect to equivariant
  bounded cohomology classes, where a class has the property if it
  admits a representative that does.
\end{defi}

By symmetry of the definitions of heads and tails, $\kappa$
is transverse to~$\Gamma s$ if and only if it is transverse to~$\Gamma
\bar{s}$

\begin{rem}
  \label{rem:notnontransverse}
  In the square grid (with the translation action by~$\Z^2$), we
  consider two ``orthogonal'' segments~$\widehat{s_1}$ and $\widehat{s_2}$
  as in Figure~\ref{fig:notnontransverse}. Then, the
  cochain~$\delta^1 f_{\widehat{s_2}}$ and the orbit~$\Z^2 \widehat{s_1}$
  are not non-transverse.
    
  Similarly, a straightforward computation shows that 
  in Example~\ref{ex:FxF}, the
  cochain~$\delta^1 f_{\widehat{s_2}}$ and the orbit~$\Gamma
  \widehat{s_1}$ are not non-transverse.
\end{rem}

\begin{figure}
  \begin{center}
    \begin{tikzpicture}[x=1cm,y=1cm,thick]
      \begin{scope}[black!50]
        \squaregrid
      \end{scope}
      \halfspace{(1,0)}{90}{\widehat{s_1}}
      \halfspace{(0,0)}{90}{}
      \def\halfspacecolor{\shalfspacecolor}
      \begin{scope}[shift={(0,-0.6)}]
      \halfspace{(0,1)}{180}{\widehat{s_2}\quad}
      \halfspace{(0,0)}{180}{}
      \end{scope}
    \end{tikzpicture}
  \end{center}

  \caption{Not non-transversality in the square grid (\cref{rem:notnontransverse}).}
  \label{fig:notnontransverse}
\end{figure}
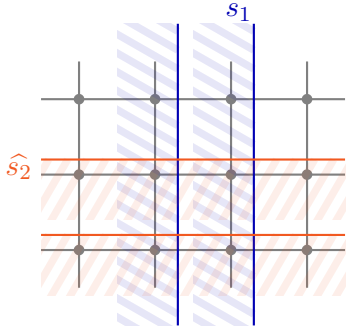

The main point of Definition \ref{defi:nontransverse} is that it
allows to define the auxiliary map~$\tilde{\kappa}$ in the proof
below, similarly to the auxiliary maps used in previous work
on cup products of quasimorphisms.

\begin{thm}\label{thm:mainccc}
  Let $\Gamma \actson X$ be an action of a group~$\Gamma$ on a
  finite-di\-men\-sion\-al $\cat0$ cube complex~$X$ with finite staircase
  length.  Let $l \in \N$, let $s \in
  X_{\h}^{(l)}$, and let $\zeta \in \HH^n_{\Gamma, b}(X; \R)$ be
  non-transverse to~$\Gamma s$. Then
  \[ [\delta^1 f_s] \cup \zeta = 0 \in \HH_{\Gamma,b}^{n+2}(X;\R).
  \]
\end{thm}

\begin{proof}
  The statement is obvious if $f_s = 0$, so we may assume that
  $\Gamma s \neq \Gamma \bar{s}$ and use Equation~(\ref{formula:fs})
  from Remark~\ref{rem:inversion}.
  Let $\kappa$ be a cocycle representing $\zeta$ that
  is non-transverse to~$\Gamma s$.
  As in the proofs of the corresponding results for free
  or surface groups~\cite{Heuer:cup, AB:cup, surface:cup},
  we provide an explicit construction of a 
  bounded cochain~$\beta \in \CC_{\Gamma,b}^{n+1}(X;\R)$ with
  \[ \delta^{n+1} \beta = (\delta^1 f_s) \cup \kappa. 
  \]
  More specifically, it suffices to find a cochain~$\eta \in
  \CC_{\Gamma,b}^{n}(X;\R)$ such that
  \[ \beta \coloneqq f_s \cup \kappa + \delta^n \eta
  \]
  is bounded.
  Indeed:
  \[
  \delta^{n+1}(f_s \cup \kappa + \delta^n \eta)
  = \delta^1(f_s) \cup \kappa - f_s \cup \delta^n(\kappa) + \delta^{n+1}(\delta^n \eta)
  = \delta^1 f_s \cup \kappa.
  \]

  \emph{Construction of~$\eta$.}
  For $t\in X_{\h}^{(l)}$ and $x_1, \ldots, x_n \in V$, we define
  \begin{align*}
    \tilde{\kappa} (t,x_1, \ldots, x_n) \coloneqq \frac{1}{2} \cdot \es(t) \cdot
    \bigl(\kappa(\alpha, x_1, \ldots, x_n) + \kappa(\omega, x_1, \ldots, x_{n})\bigr), 
  \end{align*}
  where $\alpha \in \alpha(t)$ and $\omega \in \omega(t)$.  This is
  well-defined, i.e., independent of the choices of $\alpha$
  and~$\omega$: If $t \in \Gamma s \cup \Gamma \bar{s}$, then we apply
  that $\kappa$ and $t$ are non-transverse (recall that $\kappa$ and
  $\Gamma \bar{s}$ are also non-transverse).  Otherwise, by definition
  $\es(t) = 0$ and so the whole term is~$0$.

  By construction, $\tilde{\kappa} (t,x_1, \ldots, x_n) = 0$ if
  $t\not\in \Gamma s \cup \Gamma \bar{s}$; moreover, $\|
  \tilde{\kappa}\|_\infty \leq \| \kappa \|_\infty$, and
  $\tilde{\kappa}$ satisfies the identity
  \begin{equation}
    \label{eq_symmetry_tilde_kappa_ccc}
    \tilde{\kappa} (t,x_1, \ldots, x_n) = - \tilde{\kappa} (\bar{t},x_1, \ldots, x_n).
  \end{equation}
  We now define $\eta \colon V^{n+1} \to \R$ via
  \begin{align*}
    \eta (x_0,x_1, \ldots, x_{n}) \coloneqq
    & \sum_{t\in [x_0,x_1]_{\h}^{(l)}}\tilde{\kappa}(t,x_1, \ldots, x_n).
  \end{align*}

  \emph{Boundedness of $\beta$.}
  Let $x_0, \dots, x_{n+1} \in V$. By definition, we have 
  \begin{align*}
  \beta(x_0, \ldots, x_{n+1})  
        &= f_s(x_0,x_1) \cdot \kappa(x_1, \ldots, x_{n+1}) \\
        &+ \eta(x_1, \ldots, x_{n+1}) - \eta(x_0,x_2, \ldots, x_{n+1}) \\
        &+ \sum_{j = 2}^{n+1} (-1)^{j}\cdot \eta(x_0,x_1, \ldots, \hat{x}_j, \ldots, x_{n+1}).
  \end{align*}
  Entering the definitions of the counting map~$f_s$
  (Remark~\ref{rem:inversion}) and~$\eta$, we get
  \begin{align*}
  \beta(x_0, \ldots, x_{n+1})  
	&= \sum_{t\in [x_0,x_1]_{\h}^{(l)}} \es(t) \cdot \kappa(x_1, \ldots,  x_{n+1})\\
	&+ \sum_{t\in [x_1,x_2]_{\h}^{(l)}} \tilde{\kappa}(t, x_2, \ldots, x_{n+1}) - \sum_{t\in [x_0,x_2]_{\h}^{(l)}} \tilde{\kappa}(t, x_2, \ldots, x_{n+1})\\
	&+ \sum_{j = 2}^{n+1} (-1)^{j} \sum_{t\in [x_0,x_1]_{\h}^{(l)}} \tilde{\kappa}(t, x_1, \ldots, \hat{x}_j, \ldots, x_{n+1}).
  \end{align*}
  The cochain~$\kappa$ is a cocycle. Therefore, for all~$y \in V$, we
  have
  \begin{align*}
0 	&= \delta^n \kappa(y, x_1, \ldots, x_{n+1}) \\
	&= \kappa(x_1, \ldots, x_{n+1}) - \kappa(y, x_2, \ldots, x_{n+1}) \\
	&+ \sum_{j = 2}^{n+1} (-1)^{j} \cdot \kappa(y, x_1, \ldots, \hat{x}_j, \ldots, x_{n+1}).
  \end{align*}
  Using this once for~$\alpha \in \alpha(t)$ and once for~$\omega \in
  \omega(t)$, we see that 
  \begin{alignat*}{2}
    & 2\cdot\sum_{t\in [x_0,x_1]_{\h}^{(l)}} \es(t) \cdot \kappa(x_1, \ldots,  x_{n+1})
    \\ 
    & =  \sum_{t\in [x_0,x_1]_{\h}^{(l)}} \es(t) 
    \cdot \biggl( \kappa(\alpha, x_2, \ldots, x_{n+1})
    - \sum_{j = 2}^{n+1} (-1)^{j} \kappa(\alpha, x_1, \ldots, \hat{x}_j, \ldots, x_{n+1})
    \\
    & 
    \phantom{= \sum_{t\in [x_0,x_1]_{\h}^{(l)}}}
    + \kappa(\omega, x_2, \ldots, x_{n+1})
    - \sum_{j = 2}^{n+1} (-1)^{j} \kappa(\omega, x_1, \ldots, \hat{x}_j, \ldots, x_{n+1})
    \biggr)
    \\
    & = 2 \cdot \sum_{t\in [x_0,x_1]_{\h}^{(l)}}
    \biggl( \tilde{\kappa}(t,x_2, \ldots, x_{n+1})
    -\sum_{j = 2}^{n+1} (-1)^{j} \tilde{\kappa}(t,x_1, \ldots, \hat{x}_j, \ldots, x_{n+1})
    \biggr).
  \end{alignat*}
  This allows us to simplify the above to
  \begin{align*}
    \beta(x_0, \ldots, x_{n+1})   
	&= \sum_{t\in [x_0,x_1]_{\h}^{(l)}} \tilde{\kappa}(t, x_2, \ldots, x_{n+1})\\
	&+ \sum_{t\in [x_1,x_2]_{\h}^{(l)}}  \tilde{\kappa}(t, x_2, \ldots, x_{n+1}) - \sum_{t\in [x_0,x_2]_{\h}^{(l)}}  \tilde{\kappa}(t, x_2, \ldots, x_{n+1}).
  \end{align*}


  Now let $\median$ denote the median of $x_0,x_1,x_2$.  For all
  $i\not = j$, the union $[x_i,m]_{\h}^{(l)}\cup [m,x_j]_{\h}^{(l)}$
  is contained in $[x_i,x_j]_{\h}^{(l)}$ and the only segments in
  $[x_i,x_j]_{\h}^{(l)}$ that are not contained in this union are
  those that contain $m$ in their interior.  There are at most $(l-1) \sigma d^l$
  of those, by Lemma~\ref{lem:fd:segments:fix},
  where $d = \dim(X)$ and $\sigma$ is the staircase length of~$X$.
  Hence, we can rewrite
  \begin{align*}
    \sum_{t\in [x_i,x_j]_{\h}^{(l)}} \tilde{\kappa}(t, x_2, \ldots, x_{n+1}) &=  \sum_{t\in [x_i,m]_{\h}^{(l)}} \tilde{\kappa}(t, x_2, \ldots, x_{n+1})\\
	&+ \sum_{t\in [m,x_j]_{\h}^{(l)}} \tilde{\kappa}(t, x_2, \ldots, x_{n+1}) \\
	&+ \sum_{t\in [x_i,x_j]_{\h}^{(l)} \setminus [x_i,m]_{\h}^{(l)}\cup [m,x_j]_{\h}^{(l)}} \tilde{\kappa}(t, x_2, \ldots, x_{n+1})
  \end{align*}
  and the last sum is bounded by~$(l-1) \sigma d^l
  \cdot \| \tilde{\kappa} \|_\infty \leq (l-1) \sigma d^l \cdot \| \kappa \|_\infty$.

  It follows that up to a uniformly (in~$x_0,\dots,x_{n+1}$)
  bounded error, we have
  \begin{align*}
    \beta(x_0, \ldots, x_{n+1})  
	&= \sum_{t\in [x_0,m]_{\h}^{(l)}} \tilde{\kappa}(t, x_2, \ldots, x_{n+1}) + \sum_{t\in [m,x_1]_{\h}^{(l)}} \tilde{\kappa}(t, x_2, \ldots, x_{n+1})\\
	&+ \sum_{t\in [x_1,m]_{\h}^{(l)}} \tilde{\kappa}(t, x_2, \ldots, x_{n+1}) + \sum_{t\in [m,x_2]_{\h}^{(l)}} \tilde{\kappa}(t, x_2, \ldots, x_{n+1}) \\
	&- \sum_{t\in [x_0,m]_{\h}^{(l)}} \tilde{\kappa}(t, x_2, \ldots, x_{n+1}) - \sum_{t\in [m,x_2]_{\h}^{(l)}} \tilde{\kappa}(t, x_2, \ldots, x_{n+1}).
  \end{align*}
  This is $0$ because by \cref{eq_symmetry_tilde_kappa_ccc}, we have
  \begin{equation*}
    \sum_{t\in [m,x_1]_{\h}^{(l)}} \tilde{\kappa}(t, x_2, \ldots, x_{n+1}) =  - \sum_{t\in [x_1,m]_{\h}^{(l)}} \tilde{\kappa}(t, x_2, \ldots, x_{n+1}).
  \end{equation*}
  In fact, this is the only place in the argument that uses
  the symmetry of \cref{eq_symmetry_tilde_kappa_ccc}.
\end{proof}

By applying the orbit map (Subsection \ref{ss:equivariant}), we obtain a vanishing result in bounded cohomology of groups:

\begin{cor}\label{cor:mainccc:groups}
  Let $\Gamma \actson X$ be an action of a group~$\Gamma$ on a
  finite-dimensional $\cat0$ cube complex~$X$ with finite staircase
  length.  Let $l \in \N$, let $s \in
  X_{\h}^{(l)}$, let $x \in V$ and let $\zeta \in \HH^n_{\Gamma, b}(X; \R)$ be
  non-transverse to~$\Gamma s$. Then
  \[ [\delta^1 \widehat{f_{s, x}}] \cup o_x^n(\zeta) = 0 \in \HH_{b}^{n+2}(\Gamma;\R).
  \]
\end{cor}

\begin{proof}
This follows directly from Theorem \ref{thm:mainccc} (see Remark \ref{rem:qmtogroups}).
\end{proof}

\subsection{Cup products of two median classes}
\label{ss:cupccc2}

As a special case, we consider the cup product of two median classes
$[\delta^1 f_r]$ and~$[\delta^1 f_s]$.

The non-transversality hypothesis in Theorem~\ref{thm:mainccc}
can be easily checked in this case and is closely related to the notion of (non-)transversality for halfspaces from \cref{def:transversality_halfspaces}.

\begin{thm}
  \label{thm:cuptwomedian}
  Let $\Gamma \actson X$ be an action of a group~$\Gamma$ on a
  finite-dimensional $\cat0$ cube complex~$X$ with finite staircase length.
  Let $s = ( h_1 \supset \cdots \supset h_l )$ and $r = ( k_1 \supset
  \cdots \supset k_{p} )$ be $\h$-segments in~$X$.  Suppose
  that each of the four pairs $\Gamma h_1, \Gamma k_1$; $\Gamma h_1,
  \Gamma k_{p}$; $\Gamma h_l, \Gamma k_1$; $\Gamma h_l, \Gamma
  k_{p}$ is non-transverse.  Then $\delta^1 f_r$ and~$\Gamma s$
  are non-transverse. In particular, by Theorem~\ref{thm:mainccc}:
  $$[\delta^1 f_s] \cup [\delta^1 f_r] = 0 \in \HH^4_{\Gamma, b}(X; \R).$$
\end{thm}

\begin{proof}
  Again we may assume that $\Gamma r \neq \Gamma \overline{r}$ and use
  Equation (\ref{formula:fs}) from Remark \ref{rem:inversion}.  We
  show that $\delta^1 f_r(\alpha, x_1, x_2)$ is independent of $\alpha
  \in \alpha(t)$ whenever $t \in \Gamma s$, under the assumption that
  the two pairs $\Gamma h_1, \Gamma k_1$ and $\Gamma h_1, \Gamma
  k_{p}$ are non-transverse (the proof for $\omega \in \omega(t)$ is
  analogous, using that the two remaining pairs are non-transverse).
  For this, it is enough to show that $f_r(\alpha, x)$ is independent
  of $\alpha \in \alpha(t)$ whenever $t \in \Gamma s$.  Since
  $\alpha(\gamma s) = \gamma \alpha(s)$ 
  is $\Gamma$-invariant, it suffices to show the statement for~$t =
  s$.

  So let $\alpha \in \alpha(s)$. By definition, there exists an edge
  $e$ dual to $h_1$ that has $\alpha$ as one of its endpoints and
  $\alpha \in \overline{h_1}$. Let $\alpha' \in \alpha(s)$ and let
  $e'$ be the corresponding edge. Let $m \coloneqq m(\alpha, \alpha',
  x)$. Thus, $m \in \overline{h_1}$, since $m \in [\alpha, \alpha']$
  and $\overline{h_1}$ is convex. The notation is illustrated in
  Figure~\ref{fig:cuptwomedian}.
  
  Note that $\alpha$ and $\alpha'$ belong to $\mathcal{N}(\{ h_1, \overline{h_1} \}) \cap \overline{h_1}$,
  so Lemma \ref{lem:carrier} implies that $[\alpha, \alpha']_{\h}$ intersects each of the orbits
  $\Gamma k_1$, $\Gamma \overline{k_1}$, $\Gamma k_{p}$, and
  $\Gamma \overline{k_{p}}$ only trivially.
  We obtain that the same also holds for~$[\alpha, m]_{\h} \subset
  [\alpha, \alpha']_{\h}$. This implies that every occurrence of
  $\gamma r$ or~$\gamma \overline{r}$ in~$[\alpha, x]_{\h}$
  cannot intersect~$[\alpha, m]_{\h}$, and thus $[\alpha,
    x]_{\h}^{(p)} \cap \Gamma r = [m, x]_{\h}^{(p)} \cap \Gamma
  r$.  By symmetry, the analogous statements hold upon switching
  $\alpha$ and~$\alpha'$. Thus:
  \[f_r(\alpha, x)
  = \sum\limits_{t \in [\alpha, x]_{\h}^{(p)}} \er(t)
  = \sum\limits_{t \in [m, x]_{\h}^{(p)}} \er(t)
  = \sum\limits_{t \in [\alpha', x]_{\h}^{(p)}} \er(t)
  = f_r(\alpha', x).
  \qedhere
  \]
\end{proof}

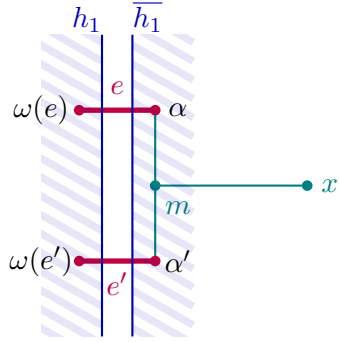
\begin{figure}
  \begin{center}
    \begin{tikzpicture}[x=1cm,y=1cm,thick]
      \begin{scope}[\geodesiccolor]
        \draw (1,-1) -- (1,1);
        \vertx{(1,0)}
        \draw (1.3,-0.3) node {$m$};
        \draw (1,0) -- (3,0);
        \vertx{(3,0)}
        \draw (3.3,0) node {$x$};
      \end{scope}
      \hedge{(0,1)}{(1,1)}
      \hedge{(0,-1)}{(1,-1)}
      \draw[\hedgecolor] (0.5,1.3) node {$e$};
      \draw (-0.5,1) node {$\omega(e)$};
      \draw (1.3,1) node {$\alpha$};
      \draw[\hedgecolor] (0.5,-1.3) node {$e'$};
      \draw (-0.5,-1) node {$\omega(e')$};
      \draw (1.3,-1) node {$\alpha'$};
      \halfspaces{(0.5,0)}{90}{h_1}
    \end{tikzpicture}
  \end{center}

  \caption{The situation in the proof of \cref{thm:cuptwomedian}.}
  \label{fig:cuptwomedian}
\end{figure}

\section{Trees}
\label{s:trees}

We specialise the results of Section~\ref{s:ccc} to the case of
actions on trees.  Fix a tree~$T = (V,E)$ and an action of a
group~$\Gamma$ on~$T$.  Seeing~$T$ as a~$\cat0$ cube complex,
halfspaces in~$T$ correspond to oriented edges, and hyperplanes to
unoriented edges.  Therefore~$\h$-segments and oriented geodesic
segments are in one-to-one correspondence, with the same value of
length.  The definition of the median quasimorphism then takes the
following more familiar form:

\begin{defi}[median quasimorphism, tree case]
  In the situation above, let $s$~be an oriented geodesic segment
  in~$T$.  We define the \emph{median quasimorphism~$f_s \colon V
    \times V \to \R$ for~$s$} for all~$(x,y) \in V \times V$ to be the
  number of translates of~$s$ in~$[x,y]$ minus the number of
  translates of~$s$ in~$[y,x]$.
\end{defi}

Once again, if $\Gamma s = \Gamma \bar{s}$ then $f_s = 0$, and
otherwise it can be computed by a formula as in
Remark~\ref{rem:inversion}~\cite[Section~3.1]{IPS}.

If $x \in V$, the quasimorphism~$f_{s,x}$ on~$\Gamma$ associated with
the quasimorphism~$f_s$ of~$\Gamma \actson T$ via the orbit map of the
base vertex~$x$ is the median quasimorphism as considered by Monod and
Shalom~\cite{Monodshalom}.  For many interesting actions, such median
quasimorphisms produce uncountably many linearly independent
quasimorphisms~\cite{IPS}\footnote{One should note that it is not entirely
clear whether conditions~(i) and~(iii) of~\cite[Theorem~1]{IPS}
actually exclude each other as the argument for this part~\cite[Section~2.2]{IPS}
is incomplete~\cite[p.~74]{fhofmann_msc}}.

\begin{example}[Brooks quasimorphisms as median quasimorphisms]
  Let $F$ be a non-abelian free group with a fixed basis, let $T$ be
  the corresponding Cayley tree.  A reduced word~$w \in F$ defines a
  segment~$s$ in~$T$ of length~$|w|$.  Then $f_{s, 1}$, the pullback
  of the median quasimorphism~$f_s$ under the orbit map at~$1$,
  coincides with the (big) Brooks quasimorphism~$H_w$ on~$F$.
\end{example}

Moreover, given a segment~$s = [x, y]$, its head and tail (Definition
\ref{defi:headstails}) are uniquely defined as~$x = \alpha(s)$, $y =
\omega(s)$.  In particular, for every segment~$s$ and every~$\kappa
\in \CC^n_{\Gamma, b}(T, \mathbb{R})$, we trivially have that $\kappa$
and~$s$ are non-transverse. Moreover, trees are finite-dimensional
and have finite staircase length. So Theorem \ref{thm:mainccc} holds
without any additional assumptions:

\begin{thm}\label{thm:maintree}
  Let $\Gamma \actson T$ be an action of a group~$\Gamma$ on a
  tree~$T$.  Let $s$ be an oriented geodesic segment in~$T$, and let
  $f_s$ be the corresponding median quasimorphism for~$\Gamma \actson
  T$. Then, for all~$n \geq 1$ and all~$\zeta \in
  \HH_{\Gamma,b}^n(T;\R)$, we have
  \[ [\delta^1 f_s] \cup \zeta = 0 \in \HH_{\Gamma,b}^{n+2}(T;\R).
  \]
\end{thm}

Similarly, Corollary \ref{cor:mainccc:groups} implies:

\begin{cor}\label{cor:maintreegroup}
  Let $\Gamma \actson T$ be an action of a group~$\Gamma$ on a tree~$T$.
  Let $x$ be a vertex, let $s$ be an oriented geodesic segment, and
  let $f_{s, x}$ be the corresponding median quasimorphism of~$\Gamma$.
  Then for every~$n \geq 1$ and every 
  class~$\zeta \in o_x^n(\HH^n_{\Gamma, b}(T; \R)) \subset
  \HH^n_b(\Gamma; \R)$, we have
  \[[\delta^1 f_{s, x}] \cup \zeta = 0.\]
  In particular, this holds in the following cases:
  \begin{enumerate}
  \item When $\zeta = [\delta^1 \widehat{f_{r, w}}]$ for some other
    segment~$r$ and vertex~$w$ in~$T$;
  \item For every~$\zeta\in \HH^n_b(\Gamma; \R)$ if all vertex
    stabilizers are amenable.
  \end{enumerate}
\end{cor}

\begin{proof}
  The last statement is the only one that is specific to trees,
  and it follows by applying Theorem~\ref{thm:amenable:stab}.
\end{proof}

\begin{rem}
\label{rem:gog}
Let $\Gamma$ be a group acting on a tree~$T$ with amenable vertex
stabilizers.  Then all edge stabilizers are also amenable.  Therefore up
to taking a subdivision of~$T$, it follows that $\Gamma$ also admits
an action on a tree \emph{without inversions} with amenable vertex
stabilizers.  In other words, $\Gamma$ is the fundamental group of a
graph of groups with amenable vertex groups.
\end{rem}

\section{Right-angled Artin groups}\label{s:raag}

We derive vanishing results for cup products of median classes of
right-angled Artin groups, considering the action on the universal
covering of their Salvetti complexes. Moreover, we give examples of
median classes of right-angled Artin groups that are non-trivial in
bounded cohomology.

\subsection{RAAGs and the Salvetti complex}

We first recall basic definitions and refer the reader to
the literature~\cite{Cha:introductionrightangled}
for more details.

If $G$ is a finite unoriented simplicial graph, the corresponding
\emph{right-angled Artin group (RAAG)}~$\Gamma \coloneqq A(G)$ is the group
with presentation
\[ \bigl\langle V(G) \mid [v, w] = 1 : \{v, w\} \in E(G)
   \bigr\rangle.
\]
The corresponding presentation complex is a square complex, and gluing
higher-dimensional cubes for each clique in~$G$ provides a compact
non-positively curved cube complex, called the \emph{Salvetti
  complex}~$S(G)$.  The universal covering~$\tilde{S}(G)$ of~$S(G)$ is
a finite-dimensional $\cat0$ cube complex, on which $\Gamma$ acts
simplicially, freely, and cocompactly.  We will denote it simply
by~$X$, and let $V$ be its vertex set. Because the action is free,
the orbit maps induce isomorphisms~$\HH^*_{\Gamma, b}(X; \R) \cong
\HH^*_b(\Gamma; \R)$ (Theorem~\ref{thm:amenable:stab}).
Moreover, $S(G)$ is a model of~$K(\Gamma,1)$ (we will not use this
fact directly). More importantly, let us recall that $S(G)$ satisfies the relevant finiteness conditions:

\begin{rem}\label{rem:salvettistaircase}
  Salvetti complexes of finite unoriented simplicial graphs are
  finite-dimensional and have
  finite staircase length~\cite[Lemma~4.17]{elia}.
\end{rem}

Given an edge~$e$ of~$X$, the projection of~$e$ to~$S(G)$ is one of
the edges of the original presentation complex, and so has a
label~$\lambda(e) \in V(G)$.  This label is preserved by the action
of~$\Gamma$. Moreover, if $e$ is dual to a halfspace~$h$, then
$\lambda(h) \coloneqq \lambda(e)$ is well-defined (this can be seen
from an alternative description
of~$\tilde{S}(G)$~\cite[Section~3.6]{Cha:introductionrightangled}).

In the following, we will always consider this setup:

\begin{setup}\label{setup:raag}
  Let $\Gamma$ be the RAAG associated with a finite unoriented
  simplicial graph~$G$ and let $X \coloneqq\tilde S(G)$ be the
  universal covering of the Salvetti complex.
\end{setup}

The following lemma collects several properties given in Setup~\ref{setup:raag}.

\begin{lemma}
\label{lem:basic_properties_RAAGs}
  In the situation of Setup~\ref{setup:raag}, we have:
  \begin{enumerate}
  \item \label{lem:transverse_Salvetti}
  	Let $h$, $k$ be halfspaces in~$X$. Then the following are equivalent:
  	\begin{enumerate}
  	\item There exists~$\gamma \in \Gamma$ with~$h \pitchfork \gamma
    		k$.
  	\item The vertices $\lambda(h)$ and~$\lambda(k)$ of~$G$ are distinct
    	and connected by an edge.
  	\end{enumerate}
  \item  \label{lem:tight_implies_non_commuting}
    Let $h,k$ be halfspaces of~$X$ such that $h\supset k$ is tightly nested. Then $\ls \lambda(h), \lambda(k) \rs$ is independent in~$G$.
	
	\item   \label{lem:normal_form_RAAGs}
  	Let $\gamma\in \Gamma$ be represented by a reduced word~$w$ in~$V(G)^\pm$ such that every two consecutive letters of~$w$ are either equal or do not commute with each other. Then for all~$x\in X(0)$, there is a unique geodesic in~$X$ from~$x$ to~$\gamma x$ and the sequence of labels on the edges occurring in this geodesic is given by~$w$.
  \end{enumerate}
 
\end{lemma}
\begin{proof}
To prove \cref{lem:transverse_Salvetti}, we first assume that $\lambda(h)$ and $\lambda(k)$ are distinct and connected by
  an edge.
  Then there is square in~$X$ whose boundary consists of two edges
  with label~$\lambda(h)$ and two edges with label~$\lambda(k)$. Let
  $e_h$ and $e_k$ be such edges with~$\lambda(e_h) = \lambda(h)$
  and $\lambda(e_k) = \lambda(k)$, respectively. By construction,
  the hyperplanes dual to these edges are transverse to each other.
  As $X$ is the universal covering of~$S(G)$, there is exactly one
  $\Gamma$-orbit of edges for each vertex in~$v\in V(G)$; such an orbit
  consists of all edges~$e$ in~$X$ with label~$\lambda(e)=v$. This
  implies that there are~$\gamma_1, \gamma_2 \in \Gamma$ such that~$\ls
  \gamma_1 h, \overline{\gamma_1 h} \rs$ is the hyperplane dual to the
  edge~$e_h$ and such that $\ls \gamma_2 k, \overline{\gamma_2 k}\rs $
  is the hyperplane dual to the edge~$e_k$. In particular, we then
  have $\gamma_1 h \pitchfork \gamma_2 k$, so $h \pitchfork
  \gamma_1^{-1}\gamma_2 k$.

  To show the reverse implication of \cref{lem:transverse_Salvetti}, let $h$, $k$ be
  halfspaces in~$X$ and $\gamma\in \Gamma$ such that $h \pitchfork
  \gamma k$. By \cref{lem:quadrangle_from_transversality}, there is a quadrangle~$e_{H}$, $f_{H}$,
  $e_{\gamma K}$, $f_{\gamma K}$ 
	such that $e_{H}$, $f_{H}$ are dual to the
    hyperplane~$H\coloneqq \ls h, \overline{h} \rs$ and such that
    $e_{\gamma K}$, $f_{\gamma K}$ are dual to the hyperplane~$\gamma K\coloneqq \ls \gamma k,
    \gamma\overline{k} \rs$.  As $X$ is
  $\cat0$, this quadrangle forms the boundary of a 2-cube
  in~$X$~\cite[Lecture 2]{Sageev2014}.  This implies that $\lambda(H)$
  and $\lambda(\gamma K) = \lambda(K)$ form a clique in~$G$, i.e.,
  these vertices are connected by an edge.

Regarding \cref{lem:tight_implies_non_commuting},
  Fern{\'o}s--Forrester--Tao show that if $h\supset k$ is tightly
  nested, then there is no~$\gamma\in \Gamma$ such that $h\pitchfork
  \gamma h$ \cite[Lemma 7.3, cf.~Definition 7.1.(2)]{FFT}. By
  \cref{lem:transverse_Salvetti}, this implies that $\ls \lambda(h),
  \lambda(k) \rs$ is an independent set in~$G$ (consisting of one or
  two vertices).

Item~\ref{lem:normal_form_RAAGs} follows from canonical forms of words in
 RAAGs~\cite[Section~2.3]{Cha:introductionrightangled}.
\end{proof}

\subsection{Vanishing cup products on RAAGs}

We apply Theorem~\ref{thm:mainccc} and Theorem~\ref{thm:cuptwomedian}
to RAAGs.

\begin{cor}
  \label{cor:RAAGs}
  In the situation of Setup~\ref{setup:raag}, let $l \in \N$ and let
  $s = ( h_1 \supset \cdots \supset h_l ) \in X_{\h}^{(l)}$.  If
  $\lambda(h_1)$ and $\lambda(h_l)$ are isolated vertices of~$G$, then
  every~$\kappa \in \CC^n_{\Gamma, b}(V; \R)$ is non-transverse
  to~$\Gamma s$. In particular, for every~$\zeta \in \HH^n_{\Gamma,
    b}(X; \R)$, we have
  \[ [\delta^1 f_s] \cup \zeta
   = 0 \in \HH^{n+2}_{\Gamma, b}(X; \R).
  \]
\end{cor}
\begin{proof}
  The finiteness conditions on~$X$ are satisfied
  (Remark~\ref{rem:salvettistaircase}).
  
  As in the case of trees, non-transversality is automatic if heads
  and tails of segments are unique.  The highest dimension of a cube
  containing an edge~$e$ of~$X$ equals the highest cardinality of a
  clique in~$G$ containing~$\lambda(e)$.  In particular, if
  $\lambda(e)$ is isolated in~$G$, then by definition, no edge
  in~$S(G)$ with label~$\lambda(e)$ is contained in a
  higher-dimensional cube. Hence the same is true for all such edges
  in the universal covering~$X$, which implies that each of the two
  halfspaces dual to~$e$ has a unique head and a unique
  tail. Moreover, this property is preserved by the action
  of~$\Gamma$. This proves the non-transversality claim.

  We can then apply Theorem~\ref{thm:mainccc} to obtain the vanishing
  result.
\end{proof}

Applying Theorem~\ref{thm:cuptwomedian} to the setting of
RAAGs, we obtain:

\begin{cor}
  \label{cor:RAAG:twomedian}
  In the situation of Setup~\ref{setup:raag}, let $l \in \N$,  
  let $s = ( h_1 \supset \cdots \supset h_l ) \in X_{\h}^{(l)}$, and $r
  = ( k_1 \supset \cdots \supset k_{p} ) \in X_{\h}^{(p)}$.
  Suppose that each of the four pairs
  $\lambda(h_1), \lambda(k_1)$; $ \lambda(h_1), \lambda(k_{p})$; $
  \lambda(h_l), \lambda(k_1) $; $ \lambda(h_l), \lambda(k_{p})$ are
  \emph{not} connected by an edge in~$G$.  Then
  \[ [\delta^1 f_s] \cup [\delta^1 f_r]
   = 0 \in \HH^4_{\Gamma, b}(X; \R).
  \]
\end{cor}
\begin{proof}
  The finiteness conditions on~$X$ are satisfied
  (Remark~\ref{rem:salvettistaircase}). 
  The characterisation of \cref{lem:transverse_Salvetti} of \cref{lem:basic_properties_RAAGs} shows
  that in this situation the non-transversality hypothesis of
  Theorem~\ref{thm:cuptwomedian} is satisfied. Applying
  Theorem~\ref{thm:cuptwomedian} to the action of~$\Gamma$ on~$X$ proves the claim.
\end{proof}

Again, let us remark that although the definition of
non-transversality is more natural at the level of the action, the
orbit map gives a canonical isomorphism $\HH^n_{\Gamma, b}(X; \R)
\cong \HH^n_b(\Gamma; \R)$, and so Corollaries \ref{cor:RAAGs}
and~\ref{cor:RAAG:twomedian} are actually statements about the usual
bounded cohomology of RAAGs.

\subsection{Non-trivial median classes in RAAGs}\label{ss:nontrivialmedian}

Finally, we show that there is a wealth of median classes
on RAAGs that are non-trivial in bounded cohomology.

Every halfspace in~$X$ induces an orientation on its dual edges.
Now every \emph{oriented} edge in~$X$
determines an element in
\begin{equation*}
  V(G)^\pm = \bigl\{ v  \bigm| v \in V(G) \bigr\}
  \cup \bigl\{ v^{-1} \bigm| v \in V(G) \bigr\}.
\end{equation*}
Hence, an $\h$-segment~$s\in X_{\h}^{(l)}$ determines a
sequence~$\signedlabels(s)$ of such elements. It is easy to see that
this sequence is reduced, so contains no subsequence of the form~$v
v^{-1}$ or~$v^{-1}v$. Hence, $\signedlabels(s)$ is a reduced word
representing an element of~$\Gamma$.  Similarly, every geodesic in $X$
defines a reduced word in~$V(G)^\pm$. However, one should be
aware that $\lambda^\pm(x)$ does not necessarily determine
the orbit~$\Gamma s$ completely.

To show non-triviality of certain median classes, we use free
subgroups~$\Lambda$ of~$\Gamma$.  Let us recall that a set of vertices
of a graph is \emph{independent} if the graph does not contain an edge
between any two of the vertices in this set.  Whenever $F\subset
V(G)$ is independent, it spans a free subgroup~$\Lambda \coloneqq \langle F
\rangle$ of~$\Gamma$. (These are often called \emph{parabolic} or
\emph{graphical} free subgroups.)

There are two canonical ways of relating the bounded cohomology
of~$\Gamma$ with the one of~$\Lambda$: The group~$\Lambda$ is
a retract of~$\Gamma$ and so~$\HH^*_b(\Lambda;\R)$ is a retract
of~$\HH^*_b(\Gamma;\R)$). More explicitly: 
On the one hand, the
inclusion~$i\colon \Lambda \hookrightarrow \Gamma$ induces a
\emph{restriction map}
\begin{equation}
\label{eq:gamma_to_lambda}
   \HH_b^2(i;\R) \colon \HH_b^2(\Gamma;\R)\to \HH_b^2(\Lambda;\R).
\end{equation}
On the other
and, there is a canonical epimorphism~$p\colon\Gamma \to \Lambda$ that is
given by the identity on~$\Lambda \leq \Gamma$ and sends every other
generator~$v \in V(G) \setminus F$ to the identity. This induces a
\emph{pullback map}
\begin{equation}
  \label{eq:lambda_to_gamma}
  \HH_b^2(p;\R) \colon \HH_b^2(\Lambda;\R)\to \HH_b^2(\Gamma;\R).
\end{equation}
We first use the restriction map (\cref{eq:gamma_to_lambda}) to show
that there are large classes of non-trivial median classes on RAAGs to
which our results apply (\cref{prop:BFbig},
\cref{prop:fsxnontrivial}).  We later show that these classes are not
pullbacks of Brooks quasimorphisms
(\cref{prop:median_qms_are_no_brooks}).

Let $w\in \Lambda$ be a reduced word with respect to the
basis~$F$. Then as explained by Fern{\'o}s--Forrester--Tao~\cite[proof
  of Proposition~4.7]{FFT}, for every vertex~$x\in X(0)$, there exists
a sequence~$s(w)$ of tightly nested halfspaces in~$X$ that has $x \in
\alpha(s(w))$ as a head and such that $\signedlabels(s(w)) = w$.
Similarly to the non-overlapping case considered by
Fern{\'o}s--Forrester--Tao~\cite[proof of Proposition~4.7]{FFT}, one
can show that $f_{s(w),x}$ restricts to the Brooks quasimorphsim~$H_w$
on~$\Lambda$. This is in fact true for every segment $s$ that has $x$ as a head and such that $\signedlabels(s) = w$:

\begin{lemma}[{\cite[proof of Proposition~4.7]{FFT}}]
  \label{lem:restriction_is_brooks}
  In the situation of Setup~\ref{setup:raag}, let $F\subset V(G)$ be
  independent and $\Lambda \coloneqq \langle F \rangle$.  Let $x\in X(0)$ and
  $s$ an $\h$-segment that has $x$ as a head: $x \in
  \alpha(s)$. Then
  \begin{equation*}
    f_{s,x}|_{\Lambda} = H_{w},
  \end{equation*}
  where $w:= \lambda^\pm(s)$ and $H_{w}\colon\Lambda\to \R$ is the big Brooks quasimorphism
  associated to~$w$.
\end{lemma}

Now we choose~$x \in X(0)$ and consider the following subspace
of~$\HH^2_{b}(\Gamma;\R)$:
\begin{equation*}
  B_F \coloneqq
  \bigl\{ [\delta^1 \widehat{f_{s,x}}]
  \bigm| l \in \N,\ s = ( h_1 \supset \cdots \supset h_l ) \in X_{\h}^{(l)},
         \ \text{and } \lambda(h_1), \lambda(h_l) \in F
  \bigr\}
\end{equation*}
If $F$ is an independent subset of~$G$, then on the one hand, our
results show that the cup product is trivial on~$B_F$
(Corollary~\ref{cor:RAAG:twomedian}); on the other hand, we can use
\cref{lem:restriction_is_brooks} to show that the subspace~$B_F$ is
very big:

\begin{prop}\label{prop:BFbig}
  In the situation of Setup~\ref{setup:raag}, 
  if $F\subset V(G)$ is an independent set of vertices in~$G$ and $|F|
  \geq 2$, then $\dim B_F = \infty$ and for all~$\phi, \psi \in B_F$,
  we have
  \begin{equation*}
    \phi \cup \psi = 0 \in \HH^4_{b}(\Gamma; \R).
  \end{equation*}
\end{prop}
\begin{proof}
  For the infinite-dimensionality, we follow the argument from the
  non-overlapping case~\cite[Proposition~4.7]{FFT}: Let $\{a,b\}$ be a
  two-element subset of~$F$. Then, the subgroup~$\Lambda
  \subset\Gamma$ generated by~$\{a,b\}$ is free. It suffices to show
  that the restriction~$\HH_b^2(\Gamma;\R) \to \HH_b^2(\Lambda;\R)$
  maps~$B_F$ to an infinite-dimensional subspace
  of~$\HH_b^2(\Lambda;\R)$.  Let $w\in \Lambda$ a reduced word with respect to $\{a,b\}$. As noted above, there
  exists an ${\h}$-segment~$s(w)$ in~$X$ that has some vertex $x$ as a head and such that the median
  quasimorphism~$f_{s(w),x}$ on~$\Gamma$ restricts to the (big) Brooks
  quasimorphism of~$w$ on~$\Lambda$
  (\cref{lem:restriction_is_brooks}). The set of (big) Brooks
  quasimorphisms on~$\Lambda$ spans an infinite-dimensional subspace
  of~$\HH_b^2(\Lambda;\R)$~\cite{Brooks, Grigorchuk} (see also
  Proposition~\ref{prop:kercomp2}). By construction, the restriction
  maps~$B_F$ onto this subspace. Hence, $\dim B_F = \infty$.

  That the cup product vanishes on~$B_F$ is an immediate consequence
  of \cref{cor:RAAG:twomedian}.
\end{proof}

\begin{rem}
  In the case in which $F = \{ v \}$ is a single vertex that is not
  central in~$\Gamma = A(G)$, we still obtain the same result.
  Indeed, suppose that $v' \in V(G)$ is another vertex such that $F' =
  \{ v, v' \}$ is independent.  Then the previous lemma shows that
  $B_{F'}$ is infinite-dimensional, by pairing each element with a
  Brooks quasimorphism of the free group on~$\{ v, v' \}$.  If we only
  consider those Brooks quasimorphisms that start and end with~$v$,
  then we obtain elements in~$B_F$, which are still non-trivial. These
  witness that the space is infinite-dimensional~\cite{Grigorchuk}.

  It is possible to choose such a~$v'$, unless $v$ is central.  But
  the latter case is of little interest, since then the canonical
  epimorphism~$A(G) \to A(G \setminus \{ v \})$ induces an
  isomorphism~$\HH_b^*(A(G\setminus \{v\});\R) \to \HH_b^*(A(G);\R)$
  in bounded cohomology.  This follows directly by applying Gromov's
  Mapping Theorem~\cite{vbc}\cite[Section 5.5]{Frigerio}.
\end{rem}

We point out that every non-trivial element of a RAAG is detected by
the homogenisation of a median quasimorphism fitting in one of the
subspaces~$B_F$ above (\cref{prop:fsxnontrivial}). Together with
\cref{prop:median_qms_are_no_brooks} this indicates that median
quasimorphisms form a rich class of quasimorphisms on RAAGs.

Recall that if $\Gamma$ is a group, the map~$\delta^1$ induces an
isomorphism between the space of homogeneous quasi-morphisms
of~$\Gamma$ and the kernel of the comparison map (see Proposition
\ref{prop:kercomp2}).
For a quasimorphism~$f\colon \Gamma \to \R$, let $\bar{f}$ be the
\emph{homogenisation} of~$f$, i.e., the quasimorphism defined by
$\bar{f}(\gamma) \coloneqq \lim_{n\to \infty} f(\gamma^n)/n$ for
all~$\gamma \in \Gamma$.

\begin{prop}\label{prop:fsxnontrivial}
  In the situation of Setup~\ref{setup:raag}, let $\gamma\in
  \Gamma \setminus \ls 1 \rs$ and $x\in X(0)$. Then there exists an~$l
  \in \N$ and a segment~$s = ( h_1 \supset \cdots \supset h_l )\in
  X_{\h}^{(l)}$ such that $\bar{f}_{s,x}(\gamma) \geq 1$ and such that
  $\ls \lambda(h_1), \, \lambda(h_l) \rs$ is an independent set in~$G$.
\end{prop}
\begin{proof}
Call a segment~$s = ( h_1 \supset \cdots \supset h_l )$ \emph{$\gamma$-nested} if $h_l \supset \gamma h_1 $, i.e.~if $s$ and $\gamma s$ assemble to a segment
  $( h_1 \supset \cdots \supset h_l \supset \gamma h_1 \supset \cdots \supset \gamma h_l )$.
  Let $s$ be a maximal $\gamma$-nested segment in $[x, \gamma x]_{\h}$.

  F\"ohn proves~\cite[Lemma~4 and Proof of Theorem 15]{FFTsimp} that
  for all~$n \geq 1$, the number of
  non-overlapping translates of~$s$ in~$[x, \gamma^n x]_{\h}$ is at
  least~$n$ and that $[x, \gamma^n x]_{\h}$ contains no translate
  of~$\bar{s}$. This implies that $f_{s}(x,\gamma^n x)$, which is the
  number of (possibly overlapping) translates of~$s$ in~$[x,\gamma^n
    x]_{\h}$ minus the number of translates of~$\bar{s}$
  in~$[x,\gamma^n x]_{\h}$, is at least~$n$. In particular,
  $\bar{f}_{s,x} (\gamma) = \lim_{n\to \infty} f_s(x,\gamma^n x)/n
  \geq 1$.

  Furthermore, whenever $s = ( h_1 \supset \cdots \supset h_l )$ is
  maximally $\gamma$-nested, then $h_l \supset \gamma h_1$ is tightly
  nested~\cite[Lemma 12]{FFTsimp}. So
  \cref{lem:tight_implies_non_commuting} of \cref{lem:basic_properties_RAAGs} implies that $\ls
  \lambda(h_1), \lambda(h_l) \rs$ is independent.
\end{proof}

\subsection{Comparison with pullbacks of Brooks quasimorphisms}\label{ss:pullbackbrooks}

The non-trivial median classes described above are
genuinely related to the cubical (or median) structure of the $\cat0$
cube complex and not just equal to the classes that one obtains as
pullbacks of Brooks quasimorphisms via \cref{eq:lambda_to_gamma}.

\begin{prop}
  \label{prop:median_qms_are_no_brooks}
  In the situation of Setup~\ref{setup:raag}, let~$s$ be
  an~$\h$-segment, and suppose that one of the following holds:
  \begin{enumerate}
  \item \label{assumption1} The set of labels of halfspaces in~$s$ is
    not independent in $G$;
  \item \label{assumption2} The set of labels of halfspaces in~$s$ is
    an independent set $F \subset V(G)$ of size $2\leq |F| < |V(G)|$, such
    that~$\Lambda := \langle F \rangle$ is not a direct factor of
    $\Gamma$, and~$\signedlabels(s)$ is cyclically
    reduced.
  \end{enumerate}
  Then for every~$x \in X(0)$, the quasimorphism~$f_{s, x}$ is
  \emph{not} at bounded distance from the pullback of a non-zero
  Brooks quasimorphism on a parabolic free subgroup.
  \end{prop}
  
\begin{rem}
  We formulated assumption~(\ref{assumption2}) in order to have a
  comparably short proof that still witnesses that many median classes are not at
  a bounded distance from pullbacks of Brooks quasimorphisms. However
  this assumption is by no means necessary. For instance, with some
  more work, one can drop the assumption that~$\signedlabels(s)$ is
  cyclically reduced.  On the other hand, the hypothesis
  that~$\Lambda$ is not a direct factor is necessary (see
  Example~\ref{ex:FxF}).
\end{rem}
  
\begin{proof}[Proof of \cref{prop:median_qms_are_no_brooks}]
We first show (by contraposition) that Assumption~(\ref{assumption1}) implies the conclusion.
  Let~$s$ be an~$\h$-segment of~$X$ and assume that the conclusion does not hold. Then there exist~$x\in X(0)$, a parabolic subgroup~$\Lambda = \langle F \rangle$ and
  a word~$w \in \Lambda$ such that~$f_{s, x}$ is at a bounded distance
  from~$\widetilde{H_w} := H_w \circ p$, where $p : \Gamma \to
  \Lambda$ is the retraction and $H_w$ is non-zero.
  Since~$H_w|_{\Lambda}$ is unbounded,
  $f_{s, x}|_{\Lambda}$ is also unbounded. Now, if $\gamma \in
  \Lambda$, then by \cref{lem:normal_form_RAAGs} of \cref{lem:basic_properties_RAAGs} there is a unique
  geodesic from~$x$ to~$\gamma x$, and the sequence of labels on the
  edges of this geodesic describes a reduced word in~$F^\pm$
  representing~$\gamma$. This implies that every halfspace~$h$
  in~$[x,\gamma x]_{\h}$ has label~$\lambda(h)\in F$. In particular,
  $[x,\gamma x]_{\h}$ can only contain a translate of~$s$ or
  of~$\overline{s}$ if every element of~$\signedlabels(s)$ is contained
  in~$F^{\pm}$.  This shows that~$\signedlabels(s)$ describes a word
  in~$\Lambda$, so the set of labels of halfspaces of $s$ is independent in $G$.

We now show that  Assumption~(\ref{assumption2}) implies the conclusion.
  Let~$x \in X(0)$ and let~$s$ be an~$\h$-segment such the set of
  labels of halfspaces in~$s$ is an independent set~$F \subset V(G)$
  with~$|F| \geq 2$, such that $\lambda^\pm(s)$ is cyclically reduced,
  and such that $\Lambda := \langle F \rangle$ is
  not a direct factor of~$\Gamma$.

  \emph{Assume} for a contradiction that~$f_{s, x}$ is at a bounded
  distance from the pullback of a Brooks quasimorphism on a parabolic
  free subgroup~$\Lambda'$.  The proof of
  Proposition~\ref{prop:median_qms_are_no_brooks} with
  hypothesis~(\ref{assumption1}) shows that then~$\signedlabels(s) \in
  \Lambda'$.  Moreover, Lemma~\ref{lem:restriction_is_brooks} implies
  that~$f_{s,x}|_{\Lambda'} = H_{\signedlabels(s)}$, which combined
  with Lemma~\ref{lem:bounded_distance_brooks} lets us assume
  that~$\Lambda = \Lambda'$ and~$f_{s, x}$ is at a bounded distance
  from~$\widetilde{H}_w$ with~$w = \signedlabels(s) \in \Lambda$
  and~$\widetilde{H}_w = H_w \circ p$, where~$p : \Gamma \to \Lambda$
  is the retraction.

  Assumption~(\ref{assumption2}) also tells us that every element
  of~$F$ occurs in~$\lambda(s)$, and that~$\Lambda$ is not a direct
  factor of~$\Gamma$.  Therefore there exist~$v \in F$ and~$v' \in
  V(G) \setminus F$ such that $v$ and~$v'$ are not connected by an
  edge.  There is a cyclic permutation of~$w$ of the
  form~$\tilde{w} = v^{l} w v^{-l}$ with~$l\in \Z$ and such that $\tilde{w}$
  does not end with $v$ or~$v^{-1}$.
  
  Up to an automorphism we may assume that the first occurrence of~$v$
  or~$v^{-1}$ in~$\tilde{w}$ is~$v$. This implies that $w$
  neither starts nor ends with~$v^{-1}$. We may write
  \begin{equation*}
  	\tilde{w} = w_1 v^k w_2,
  \end{equation*}
  where~$k >0$, where $w_1$ does not contain $v$ or~$v^{-1}$ and $w_2$ does
  not start with $v$ or $v^{-1}$. (By our assumption on $\tilde{w}$,
  the word~$w_2$ also does not end with $v$ or~$v^{-1}$.) We note the
  following:
  \begin{equation}
  \label{eq:prop_tilde_w1} \text{If $w_1$ is non-trivial, then $\tilde{w} = w$.}
  \end{equation}
  \begin{equation}
  \label{eq:prop_tilde_w2} \text{$w_2$ is a subword of $w$.}
  \end{equation}
  Let
  \begin{equation*}
    w' \coloneqq w_1 v^{-k} v' v^{3k} v' v^{-k} w_2.
  \end{equation*}
  Then~$w'$ is a reduced expression such that each two consecutive
  letters appearing in it are either equal to one another or do not
  commute.  Moreover, since $w$ is cyclically reduced, the same holds
  for powers of~$w'$.

  We claim that for no~$n\geq 1$, the power~$(w')^n$ contains a copy of $w$
  or~$w^{-1}$: First assume that $w$ is a subword of~$(w')^n$ for some~$n \geq 1$.
  Then as $w$ does not contain~$v'$, it must be a subword
  of~$v^{-k} w_2 w_1 v^{-k}$. But as noted above, $w$ neither starts nor
  ends with~$v^{-1}$, so this implies that $w$ is a subword of~$w_2
  w_1$. This is impossible because $\len(w) = \len(\tilde{w})>
  \len(w_2 w_1)$.  Now assume that $w^{-1}$ is contained in~$(w')^n$.
  Then again, it must be contained in~$v^{-k} w_2 w_1
  v^{-k}$, which is equivalent to $w$ being a subword of~$v^{k}
  w_1^{-1} w_2^{-1} v^{k}$. If $w_1$ is non-trivial, by
  \cref{eq:prop_tilde_w1}, the word $w = \tilde{w}$ neither starts nor
  ends with~$v$. Hence, the inequality~$\len(w)> \len(w_2 w_1) =
  \len(w_1^{-1} w_2^{-1})$ again gives a contradiction. If on the
  other hand $w_1$ is trivial, then $w$ is a subword of~$v^{k}
  w_2^{-1} v^{k}$. By \cref{eq:prop_tilde_w2}, the same is true
  for~$w_2$. As $w_2$ is not a power of~$v$, this implies that $w_2$ and
  $w_2^{-1}$ overlap, which is impossible (see, e.g.,
  \cite[Lemma~3.14]{Fournier:cup}).
	
  It follows that~$f_{s, x}((w')^n) = 0$ for all~$n \in \N$. Indeed, as
  noted above, $(w')^n$ is a reduced expression such that each two
  consecutive letters appearing in it are either equal to one another
  or do not commute.  Hence, by \cref{lem:normal_form_RAAGs} of \cref{lem:basic_properties_RAAGs},
  there is a unique geodesic in~$X$ from~$x$ to~$(w')^n x$ and the
  sequence of labels of its edges is given by~$(w')^n$. This implies
  that for every $\h$-segment~$t$ contained in~$[x, (w')^n x ]_{\h}$,
  the sequence~$\signedlabels(t)$ must be a subword of~$(w')^n$.
  However, we showed in the previous paragraph that~$(w')^n$ contains no copy
  of~$w = \signedlabels(s)$ or~$w^{-1} =
  \signedlabels(\bar{s})$.
  Hence, no translate of~$s$ or~$\bar{s}$ is
  contained in in~$[x, (w')^n x ]_{\h}$.
 
  On the other hand, the epimorphism~$p\colon\Gamma \to \Lambda$
  sends~$w'$ to~$\tilde{w}$. So, we obtain~$\widetilde{H}_w((w')^n) =
  H_w(\tilde{w}^n) \geq n-1$ because~$\tilde{w}$ is a cyclic
  permutation of the letters of $w$, the word $w$ is cyclically
  reduced and no occurrence of~$w^{-1}$ can overlap with~$w$.  Hence
  $f_{s,x}$ and~$\widetilde{ H}_w$ are not at bounded distance, which
  is a contradiction.
\end{proof}

\appendix
\section{Cup products and~$\Lex$}\label{appx:lex}

Group epimorphisms induce injections on the level of bounded
cohomology in degree~$2$~\cite{bouarich2}\cite[Section 2.7]{Frigerio}.
It is not known whether this behaviour persists in higher degrees.  We
explain how vanishing results for cup products have the potential to
lead to non-examples in degree~$4$ and all degrees~$\geq 6$.

\begin{defi}
  Let $\Lex$ be the class of all groups~$\Lambda$ with the following
  property: For all groups~$\Gamma$ and all epimorphisms~$f \colon
  \Gamma \longrightarrow \Lambda$, the induced map~$\HH^*_b(f;\R) \colon
  \HH^*_b(\Lambda;\R) \longrightarrow \HH^*_b(\Gamma;\R)$
  is injective.
\end{defi}

Bouarich shows that this class~$\Lex$ contains many geometrically
defined groups (e.g., free and surface groups) and that $\Lex$ is closed with
respect to several ``amenable'' constructions~\cite{bouarichexact}.
Moreover, the class~$\Lex$ has been studied in the context of
boundedly acyclic groups~\cite{bcfp}. It seems likely that not every
group is in~$\Lex$, but no examples seem to be known.

Our recipe relating cup products to~$\Lex$ relies on the interaction
between bounded cohomology and $\ell^1$-homology. The
\emph{$\lone$-homology functor~$\HH_*^{\ell^1}(\args;\R)$} is defined as
the homology of the $\lone$-completion of the homogeneous complex on
groups (or equivalently of the $\lone$-completion of the singular
chain complex of classifying
spaces)~\cite{matsumotomorita,Loehl1}. For a group~$\Gamma$,
evaluation of bounded cocycles on $\lone$-cycles leads to a natural
bilinear map~$\langle\cdot,\cdot\rangle \colon \HH_b^*(\Gamma;\R)
\times \HH_*^{\lone}(\Gamma;\R) \to \R$.

\begin{defi}[$\lone$-detectable]
  Let $\Gamma$ be a group and $n \in \N$. A class~$\varphi \in
  \HH_b^n(\Gamma;\R)$ is \emph{$\lone$-detectable} if there exists
  an~$\alpha \in \HH_n^{\lone}(\Gamma;\R)$ with~$\langle
  \varphi,\alpha\rangle \neq 0$.
\end{defi}

\begin{prop}\label{prop:eval}
  \hfil
  \begin{enumerate}
  \item
    Cross-products of $\lone$-detectable classes are
    $\lone$-detectable and hence non-trivial.
  \item
    All non-trivial classes in~$\HH_b^2(\args;\R)$ are
    $\lone$-detectable.
  \item
    If $M$ is an oriented closed connected hyperbolic manifold
    of dimension~$n\geq 2$, then there exists an $\lone$-detectable class
    in~$\HH_b^n(\pi_1(M);\R)$.
  \end{enumerate}
\end{prop}
\begin{proof}
  The first part follows from the fact that cross products
  are multiplicative with respect to evaluation in the sense that
  \[ \langle \varphi \times \psi, \alpha \times \beta \rangle
  = \pm \langle\varphi, \alpha\rangle \cdot \langle \psi, \beta\rangle 
  \]
  holds for all bounded cohomology classes~$\varphi$, $\psi$ and all
  $\lone$-homology classes~$\alpha$, $\beta$ in compatible
  degrees~\cite[Proposition~2.5]{clara}.
  
  The second part follows from an observation by Matsumoto and
  Morita~\cite[Corollary~2.7, Theorem~2.3]{matsumotomorita}.

  The third part is a consequence of the duality principle and the
  fact that the simplicial volume of oriented closed connected
  hyperbolic manifolds is non-zero~\cite{vbc}. Thus, the dual
  fundamental class is in the image of the comparison map and the
  fundamental class witnesses $\lone$-detectability.
\end{proof}

\begin{thm}\label{thm:epi}
  Let $\Lambda_1$, $\Lambda_2$, $\Lambda_3$ be groups. We consider the
  (co)product groups
  $\Gamma  \coloneqq (\Lambda_1 * \Lambda_2) \times \Lambda_3$
  and 
  $\Lambda \coloneqq (\Lambda_1 \times \Lambda_2) \times \Lambda_3$ 
  and the group homomorphism
  \[ f \coloneqq \bigl((f_1,f_2) \circ \Delta \bigr) \times \id_{\Lambda_3}
       \colon \Gamma \longrightarrow \Lambda
  \]
  given by the projections~$f_j \colon \Gamma \longrightarrow
  \Lambda_j$ to the first and second free factor, respectively, and
  the diagonal~$\Delta \colon \Gamma \longrightarrow \Gamma \times
  \Gamma$. 
  Furthermore, we assume that there exist~$n_1, n_2, n_3 \in \N$
  with~$n_1 + n_2 + n_3 = n$ and bounded cohomology classes $\varphi_j
  \in \HH_b^{n_j}(\Lambda_j;\R)$ for each~$j \in \{1,2,3\}$ with the
  following properties:
  \begin{enumerate}
  \item[(D)]
    The classes $\varphi_1$, $\varphi_2$, $\varphi_3$ are $\lone$-detectable.
  \item[(C)] We have
    $\HH_b^{n_1}(f_1;\R) (\varphi_1) \cup \HH_b^{n_2}(f_2;\R) (\varphi_2) = 0
       \in \HH_b^{n}(\Gamma;\R).
    $ 
  \end{enumerate}
  Then $\HH_b^n(f;\R)$ is not injective. In particular, the group~$\Lambda$
  is not in~$\Lex$.
\end{thm}
\begin{proof}
  Under these hypotheses, the class
  $\varphi \coloneqq \varphi_1 \times \varphi_2 \times \varphi_3 \in \HH_b^n(\Lambda;\R)$ 
  witnesses that $\HH_b^n(f;\R)$ is not injective:
  \begin{itemize}
  \item The $\lone$-dectability~(D) and Proposition~\ref{prop:eval}
    imply $\varphi \neq 0$ in~$\HH_b^n(\Lambda;\R)$.
  \item Moreover, property~(C) and the relation between the
    cup and the cross product allow us to compute that
    \begin{align*}
      & 
      \HH_b^{n_1+n_2}\bigl((f_1, f_2) \circ \Delta;\R\bigr)
                      (\varphi_1 \times \varphi_2)
      \\
      & = \HH_b^{n_1+n_2}(\Delta;\R) \bigl(\HH_b^{n_1}(f_1;\R)(\varphi_1)
                     \times \HH_b^{n_2}(f_2;\R)(\varphi_2)\bigr)
      \\
      & = \HH_b^{n_1}(f_1;\R)(\varphi_1) \cup \HH_b^{n_2}(f_2;\R)(\varphi_2)
      \\
      & = 0.
    \end{align*}
    In particular, we obtain
    \begin{align*}
      \HH_b^n(f;\R)(\varphi)
      & = \HH_b^{n_1+n_2}\bigl((f_1, f_2) \circ \Delta;\R\bigr)
                     (\varphi_1 \times \varphi_2)
          \times \HH_b^{n_3}(\id_{\Lambda_3};\R)(\varphi_3)
      \\
      & = 0 \times \varphi_3
        = 0. 
    \end{align*}
  \end{itemize}
  Therefore, $\HH_b^n(f;\R)\colon \HH_b^n(\Lambda;\R) \longrightarrow
  \HH_b^n(\Gamma;\R)$ is \emph{not} injective.

  By construction, $f$ is an epimorphism. Therefore, the non-injectivity
  of~$\HH_b^n(f;\R)$ shows that $\Lambda$ is not in~$\Lex$.
\end{proof}

What are candidates for finding groups and classes as in
Theorem~\ref{thm:epi}\;? 
It is tempting to consider the following situation:
Let $n \in \{ 4 \} \cup \N_{\geq 6}$. We set~$n_1 \coloneqq 2$, $n_2 \coloneqq 2$,
and $n_3 \coloneqq n-4$.
\begin{itemize}
  \item
    If $n=4$, then we take~$\Lambda_3$ as the trivial group; then, 
    every~$\varphi_3 \in \HH_b^0(\Lambda_3 ;\R) \setminus \{0\}$ is
    $\lone$-detectable.

    If $n \geq 6$, then we choose $\Lambda_3$ as the fundamental group
    of an oriented closed connected hyperbolic $(n-4)$-manifold (such
    manifolds do exist); in view of
    Proposition~\ref{prop:eval}, there then also exists an
    $\lone$-detectable class~$\varphi_3 \in \HH_b^{n_3}(\Lambda_3;\R)$.
  \item Let $\Lambda_1 \coloneqq F_2$, $\Lambda_2 \coloneqq F_2$ and
    let $\varphi_1, \varphi_2 \in \HH_b^2(F_2;\R)$ be classes
    corresponding to non-trivial quasi-morphisms, e.g.,
    Brooks quasimorphisms or Rolli quasimorphisms.
    Then $\varphi_1$ and $\varphi_2$ are $\lone$-detectable
    (Proposition~\ref{prop:eval}).

    What about property~(C) from Theorem~\ref{thm:epi}\;? For~$j \in
    \{1,2\}$, the pull-back~$\HH_b^2(f_j;\R)(\varphi_j) \in \HH_b^4(F_2 *
    F_2;\R)$ unfortunately does \emph{not} obviously fall into one of
    the classes for which it is known that the cup products vanish
    (see, e.g., Proposition~\ref{prop:median_qms_are_no_brooks}).
\end{itemize}
In summary, if $\HH_b^2(f_1;\R)(\varphi_1) \cup \HH_b^2(f_2;\R)(\varphi_2)
\in \HH_b^4(F_4;\R)$ were known to be zero, then we would have a rich
class of examples of non-$\Lex$ groups.

 
\bibliographystyle{abbrv}
\bibliography{cuptreebib}

\end{document}